\documentclass[reqno,10pt]{amsart}
\usepackage{amsfonts, amsmath, amssymb, amsthm, color}

\allowdisplaybreaks[1]

\newtheorem{thm}{Theorem}[section]
\newtheorem{lemma}{Lemma}[section]
\newtheorem{prop}{Proposition}[section]

\theoremstyle{definition}
 \numberwithin{equation}{section}

\newcommand{\rr}{\mathbb{R}}
\newcommand{\al}{\alpha}
\newcommand{\de}{\delta}
 \newcommand{\eps}{\epsilon}
\newcommand{\la}{\lambda}

 \renewcommand{\(}{\left(}
\renewcommand{\)}{\right)}
\renewcommand{\[}{\left[}
\renewcommand{\]}{\right]}

\begin{document}
\title[Multiple blow-up phenomena for the sinh-Poisson equation]{Multiple blow-up phenomena for the sinh-Poisson equation}
\author{Massimo Grossi}
\address[Massimo Grossi] {Dipartimento di Matematica "G.Castelnuovo", Universit\`{a} di Roma ``La Sapienza", P.le Aldo Moro 5, 00185 Roma, Italy}
\email{grossi@mat.uniroma1.it}

\author{Angela Pistoia}
\address[Angela Pistoia] {Dipartimento SBAI, Universit\`{a} di Roma ``La Sapienza", via Antonio Scarpa 16, 00161 Roma, Italy}
\email{pistoia@dmmm.uniroma1.it}

\begin{abstract}
We consider the sinh-Poisson equation
$$(P)_\lambda\quad -\Delta u=\la\sinh u\ \hbox{in}\ \Omega,\  u=0\ \hbox{on}\ \partial\Omega,$$
where   $\Omega$ is a smooth  bounded domain in $\rr^2$ and
$\lambda$ is a small positive parameter.

If  $0\in\Omega$  and $\Omega$ is symmetric with respect to the origin, for any integer $k$ if $\la$ is small enough, we construct a family of solutions
to $(P)_\la$ which blows-up at the origin whose positive mass is $4\pi k(k-1)$ and negative mass is $4\pi k(k+1).$

It gives a complete answer to an open problem formulated by Jost-Wang-Ye-Zhou in [Calc. Var. PDE (2008) 31: 263-276].

 \end{abstract}

\date{\today}

 \maketitle

\section{Introduction}
In this paper we will study the semilinear elliptic equation
 \begin{equation}\label{p}
-\Delta u=\lambda \(e^u-e^{-u}\)\quad \hbox{in}\ \Omega,\qquad  u=0\quad
\hbox{on}\ \partial\Omega,
\end{equation}
where   $\Omega$ is a smooth bounded domain in $\rr^2$ and
$\lambda$ is a small positive parameter.

This problem arises in plasma physics and statistical mechanics. See, for instance,
 Chorin \cite{c},  Marchioro-Pulvirenti \cite{mapu} and the references therein.

This problem also plays a very important role in the study of the construction of constant mean curvature surfaces initiated by Wente \cite{w1,w2}.

In 1988 Spruck \cite{s} studied \eqref{p} when $\Omega$ contains the origin and it is a domain symmetric with respect to reflections about the $x_1$ and $x_2$ axes. In particular, he proved
that a sequence of nontrivial solutions $u_n$ of \eqref{p} with $\la_n\to0$  is such that
 $u_n(x)\to-2\ln|g(x)|^2$ uniformly on compact subsets of $\Omega\setminus\{0\},$ where $g$ is the symmetric conformal map of $\Omega$ onto the unit disk.
 Twenty years later Jost-Wang-Ye-Zhou \cite{jwyz} investigated the blow-up analysis of solutions to \eqref{p} and they give a more precise asymptotic behavior when the sequence of solutions $u_n$ blows-up as $\la_n\to0.$
Let us define the positive and the negative blow-up set of the sequence $u_n$ respectively by
 $$S_+:=\left\{x\in\Omega\ :\ \exists\ x_n\to x\ s.t.\ u_n(x_n)\to+\infty\right\}$$
 $$S_-:=\left\{x\in\Omega\ :\ \exists\ x_n\to x\ s.t.\ u_n(x_n)\to-\infty\right\}.$$
For any $x_0\in S_+\cup S_-$ let us define the positive and the negative mass of $x_0$ respectively by
$$m_+(x_0):=\lim\limits_{r\to0}\lim\limits_n\int\limits_{B(x_0,r)}\lambda_ne^{u_n(x)}dx,\ m_-(x_0):=\lim\limits_{r\to0}\lim\limits_n\int\limits_{B(x_0,r)}\lambda_ne^{-u_n(x)}dx.$$
Jost-Wang-Ye-Zhou \cite{jwyz} proved that $S_+$ and $S_-$ are finite sets and that the masses $m_+(x_0)$ and $m_-(x_0)$ are multiple of $8\pi.$
This is an analogue of the result of Li-Shafrir for the Gelfand problem
\begin{equation}\label{mean}
-\Delta u=\lambda  e^u \quad \hbox{in}\ \Omega,\qquad  u=0\quad
\hbox{on}\ \partial\Omega.
\end{equation}
In view of the relationship established by Ohtsuka-Suzuki \cite{os} 
$$\(m_+(x_0)-m_-(x_0)\)^2=8\pi \(m_+(x_0)+m_-(x_0)\)$$
it follows that for any $x_0\in S_+\cup S_-$
$$m_+(x_0)=4\pi k(k-1)\ \hbox{and}\ m_-(x_0)=4\pi k(k+1)$$
or
$$m_+(x_0)=4\pi k(k+1)\ \hbox{and}\ m_-(x_0)=4\pi k(k-1)$$
for some integers $k\ge1.$
When $k=1$ we say that $x_0$ is a simple (positive or negative) blow-up point, while if $k\ge2$ we say that $x_0$ is a multiple (nodal) blow-up point.

 Bartolucci-Pistoia \cite{bp} and Bartsch-Pistoia-Weth \cite{bpw} constructed sign-changing solutions to \eqref{p} with  one or more simple  positive and  simple negative blow-up points.
The solutions they found are sum of standard bubbles which solve the 
Liouville problem
\begin{equation}\label{plims}
-\Delta w= e^w\quad \hbox{in}\quad \rr^2,\qquad
\int\limits_{\rr^2} e^{w(x)}dx<+\infty .
\end{equation}

As far as it concerns existence of solutions with multiple blow-up points,
in \cite{jwyz} the authors asked the following question.

{\em (Q)\quad Is it   possible to find solutions to problem \eqref{p} with a  multiple nodal blow-up point, i.e. $k\ge2.$}

In this paper we give a positive answer to this question. The result we have is

\begin{thm}\label{main}
Assume   $0\in\Omega$  and $\Omega$ is symmetric with respect to the
origin, i.e. $x\in\Omega$ iff $-x\in\Omega.$

For any integer $k$, there exists $\lambda_k>0$ such that for any $\lambda\in(0,\lambda_k)$ problem \eqref{p}
has a sign-changing solution $u_\la$ such that $u_\la(x)=u_\la(-x)$ and 
\begin{equation}\label{a1}
u_\la(x)\to (-1)^k8\pi  kG(x,0)\ \hbox{uniformly on compact subsets of $\Omega\setminus\{0\}$ as $\la\to0.$} 
\end{equation} 
Moreover, the origin is a multiple nodal blow-up point   whose
  blow up values    are
\begin{equation}\label{a2}
m_-(0)=4\pi k(k+1)\ \hbox{and}\ m_+(0)=4\pi k(k-1)\  \hbox{if $k$ is even}
\end{equation} 
and
\begin{equation}\label{a3}m_-(0)=4\pi k(k-1)\ \hbox{and}\ m_+(0)=4\pi k(k+1)\  \hbox{if $k$ is odd.}
\end{equation} 
\end{thm} 

Here 
\begin{equation}\label{green}
G(x,y)={1\over 2\pi}\ln{1\over |x-y|}+H(x,y),\quad x,y\in\Omega
\end{equation}
is the Green's function of the Dirichlet Laplacian in $\Omega$ and $H(x,y)$ is its regular part.

The solution $u_\la$ is constructed by superposing $k$ different kind of bubbles  with alternating sign. Each bubble  solves a different singular
Liouville problem
\begin{equation}\label{plim}
-\Delta w=|x|^{\al-2}e^w\quad \hbox{in}\quad \rr^2,\qquad
\int\limits_{\rr^2} |x|^{\al-2}e^{w(x)}dx<+\infty 
\end{equation}
for a suitable choice of $k$ different $\al$'s (see \eqref{alfa}). The choice of $\al$'s is a crucial point in the construction of the solution. We will show in Section \ref{uno} that necessarily
$$\al_i=4i-2\quad\hbox{for any }\ i=1,\dots,k.$$

We remark that when  $\al=2$   problem \eqref{plim} reduces to the well known Liouville equation\eqref{plims} whose solutions have been classified by Chen-Li \cite{cl} to be radially symmetric. When $\al>2$ is an integer all solutions to \eqref{plim} have been classified by Prajapat-Tarantello \cite{pt}. In this case problem 
 \eqref{plim} has radial and non-radial solutions. Our construction just relies on  the radial ones.
 
Even if the solution  we find resembles a tower of bubbles, it is important to point out that it is a new kind of tower of bubbles.
Indeed, classical tower of bubbles are constructed by superposing bubbles which solve  the same limit problem in the whole space, while our solution
is constructed by superposing different bubbles which are solutions to  different limit problems in $\rr^2.$ 

This is a new phenomena:
the solution we find is generated  by cooking up  
bubbles related to different limit problems. The existence of this new kind of solutions was suggested by a recent result due to Grossi-Grumiau-Pacella \cite{ggp}. They study the asymptotic behavior of the least energy nodal radial solution  to the problem 
$$-\Delta u=|u|^{p-1}u \quad \hbox{in}\ B,\qquad  u=0\quad
\hbox{on}\ \partial B,$$
where $B$ is the unit ball in $\rr^2$ and the exponent $p$ goes to $+\infty.$ In particular, they prove that the positive  and the negative parts of   this solution (suitable scaled)
converge to the solutions of the limit problems \eqref{plim} with two different values of $\alpha$'s.

We recall that   classical towers of bubbles were constructed for some critical problems in $\rr^n$ with $n\ge3.$ In particular, towers of positive bubbles were found by 
Del  Pino-Dolbeault-Musso \cite{ddm1,ddm2},   Ge-Jing-Pacard \cite{gjp} and
 Del  Pino-Musso-Pistoia \cite{dmp}, while  towers of sign-changing bubbles were built by 
 Pistoia-Weth \cite{pw}, Musso-Pistoia \cite{mp1,mp2} and 
  Ge-Musso-Pistoia \cite{gmp}. See also Esposito-Wei \cite{ew} for a related problem with  Neumann boundary condition and Del Pino-Dolbeault-Musso \cite{ddm}
  for a problem with the $p-$Laplacian operator.

We want to emphasize that in the present paper the idea of  using bubbles related to different limit problems is crucial! 
Indeed, the proof could not work if we argue  as in all the previous papers,  where the same bubbles always  is used to build the solution.

We also want to point out that an extremely delicate point in the paper concerns the linear theory developed in Section \ref{tre}. In this framework some new ideas are necessary. Moreover,
we remark that our approach also simplifies the linear theory studied in \cite{egp} and \cite{dkm}.


Finally, we believe that Theorem \ref{main} holds even if we drop the assumption on the symmetry of $\Omega.$ More precisely, we conjecture that in any domain $\Omega$
it is possible to construct a family of sign-changing solutions which blows-up at the  maximum point of the Robin's function  with the prescribed blow-up values.

The proof of our result relies on a contraction mapping argument. The paper is organized as follows. In Section \ref{uno} we establish some preliminary estimates.
In Section \ref{due} we estimate the error term. In Section \ref{tre} we study a linear problem. In Section \ref{quattro} we complete the proof of Theorem \ref{main}. In Appendix we write some useful facts.

\medskip

{\em Acknowledgments.} 
The authors would like to thank   F. Pacella for many helpful discussions.
\medskip

\section{The ansatz and the choice of $\al$'s }\label{uno}

Let $\al\ge2.$ Let us introduce the functions
\begin{equation}\label{walfa}
w^\al_\de(x):=\ln 2\al^2{\de^\al\over\(\de^\al+|x|^\al\)^2}\quad
x\in\rr^2,\ \de>0
\end{equation}
which solve the problem \eqref{plim}.

Let us introduce the projection  $P  u$ of  a function $u$
into $H^1_0(\Omega),$ i.e.
\begin{equation}\label{pro}
 \Delta P u=\Delta u\quad \hbox{in}\ \Omega,\qquad  P u=0\quad \hbox{on}\ \partial\Omega.
\end{equation}

Let $k$ be a fixed integer. We look for a sign changing solution
to \eqref{p} as
\begin{equation}\label{ans}
u_\la(x):=W_\la(x)+\phi_\la(x),\quad W_\la(x):=\sum\limits_{i=1}^k
(-1)^iP w_{\de_i}^{\al_i}(x)
\end{equation}
where  for any $i=1,\dots,k$ the $\al_i$'s satisfy
\begin{equation}\label{alfa}
\al_i:=4i-2
\end{equation}
and the concentration parameters satisfy
\begin{equation}\label{para}
\de_i:=d_i\la^{2(k-i)+1\over\al_i}=d_i\la^{2(k-i)+1\over4i-2}\ \hbox{for some}\ d_i>0.
\end{equation}
 It is important to point out that   by \eqref{alfa} and \eqref{para} we deduce
\begin{equation}\label{para2}
  {\de_i\over\de_{i+1}}= {d_i\over d_{i+1}}\la^{2k\over 4i^2-1}\to 0\ \hbox{as}\ \la\to0.
\end{equation}

The rest term $\phi_\la$ will be choose in the space
$ \mathrm{H}^1_0(\Omega)$ and will be symmetric with respect to the origin, i.e.
$\phi(x)=\phi(-x)$ for any $x\in\Omega .$

\medskip

The choice of $\de_i$'s and $\al_i$'s    is motivated by the need for the interaction among bubbles to be small.
Indeed, an important feature is that each bubble interacts with all the other ones and  in general the interaction is not negligible!
The interaction will be measured in Lemma \ref{errore} using the function 
\begin{align}\label{tetaj}
\Theta_j(y):=&(-1)^j W_\la (\de_j y)-w_j(\de_j y)-(\al_j-2)\ln|\de_j
y| +\ln\la\nonumber \\ =&Pw_j(\de_j y)-w_j(\de_j y)-(\al_j-2)\ln|\de_j
y|+\sum\limits_{i\not= j} (-1)^{i-j}P{w_i(\de_j y)}
+\ln\la.\end{align}
The choice of parameters $\al_j$ and $\de_j$   made in \eqref{alfai} and   \eqref{deltai} (which imply \eqref{alfa} and \eqref{para}) ensures that $\Theta_j$ is small. Roughly speaking, the choice of $\al_j$ allows to kill the interaction among the $j$-th bubble and all the precedent (faster) bubbles, while the choice of $\de_j$ allows to kill the interaction among 
the $j$-th bubble and all the consecutive (slower) bubbles. 
More precisely, in order to have $\Theta_j$ small in Lemma \ref{teta} we will need to   choose   $\de_j$'s and $\al_j$'s so that
\begin{equation}\label{alfai}
(\al_j-2)+2\sum\limits^k_{i=1\atop i<j} (-1)^{i-j} \al_i=0
\end{equation}
and
\begin{equation}\label{deltai}
-  \al_j\ln\de_j -2\sum\limits^k_{i=1\atop i>j}  (-1)^{i-j} \al_i
\ln\de_i-\ln(2\al_j^2)+\sum\limits_{i=1}^k(-1)^{i-j}h_i(0)+\ln\la=0,
\end{equation}
where we agree that if $j=1$ or $j=k$ the sum over the indices $i<j$ or $i>j$ is zero, respectively.
Here $h_i(x):=4\pi\al_i H(x,0).$

By \eqref{alfai} we immediately deduce that
\begin{equation}\label{alfa1}\al_1 =2\quad \hbox{and}\quad \al_{j+1} =\al_j+4\ \hbox{for}\ j=2,\dots,k-1,\end{equation}
which implies \eqref{alfa} and by \eqref{deltai} we immediately
deduce that
\begin{equation}\label{delta2}\de_k^{\al_k} ={e^{-\sum\limits_{i=1}^k(-1)^{i-k}h_i(0)}\over 2\al_k^2}\la  \end{equation}
and
\begin{equation}\label{delta3}\de_{j }^{\al _{j }}=\(4\al_j^2\al_{j+1}^2e^{-2\sum\limits_{i=1}^k(-1)^{i-j}h_i(0)}\)\de_{j+1}^{\al_{j+1}}\la^2\ \hbox{for}\ j=1,\dots,k-1,\end{equation}
which implies \eqref{para}.
We also remark that
\begin{align}\label{hi}
 \sum\limits_{i=1}^k(-1)^{i-j}h_i(0)&=(-1)^{-j}4\pi H(0,0) \sum\limits_{i=1}^k(-1)^{i }\al_i = (-1)^{ k-j}8k\pi H(0,0),
\end{align}
 because by \eqref{alfa} we easily deduce
\begin{align}\label{somma}\sum\limits_{i=1}^k(-1)^{i }\al_i =(-1)^k2k.\end{align}

 \medskip
In order to estimate $\Theta_j$ we need to 
  introduce the following set of shrinking annulus.

For any $j=1,\dots,k$ we   set
\begin{equation}\label{anelli}
 A_j:=\left\{x\in\Omega\ :\ \sqrt {\de_{j-1}\de_j}\le |x|\le\sqrt {\de_j\de_{j+1}}\right\},\ j=1,\dots,k
\end{equation}
where we set $\de_0:=0$ and $\de_{k+1}:=+\infty.$

We point out that if $j,\ell=1,\dots,k$
$${A_j\over\de_\ell}=\left\{y \in{\Omega\over\de_\ell}\ :\ {\sqrt {\de_{j-1}\de_j}\over\de_\ell}\le |y|\le{\sqrt {\de_j\de_{j+1}}\over\de_\ell}\right\}$$
and so {\em roughly speaking}
 ${A_j\over\de_\ell}$ shrinks to the origin if $\ell<j,$ ${A_j\over\de_j}$ invades the whole space $\rr^2$  and
 $ {A_j\over\de_\ell}$ runs off to infinity if $\ell>j.$

For sake of simplicity, we set $w_i:= w_{\de_i}^{\al_i}(x) .$ 
By the maximum principle we easily deduce that
\begin{lemma}\label{pwi-lem}
\begin{align}\label{pro-exp}
 P  w_i(x)=& w_i(x)-\ln\(2\al_i^2\de_i^{\al_i}\)+h_i(x)+O\(\de_i^{\al_i}\)\nonumber\\ =&-2\ln\(\de_i^{\al_i}+|x|^{\al_i}\)+h_i(x)+O\(\de_i^{\al_i}\)  
\end{align}
and for any $i,j=1,\dots,k$
\begin{equation}\label{pwi}
  Pw_i(\de_j y)=\left\{\begin{aligned}
  &-2\al_i\ln\(\de_j|y|\) +h_i(0)\\
  &\qquad +O\({1\over|y|^{\al_i}}\({\de_i\over\de_j}\)^{\al_i}\)+O\(\de_j|y|\)+O\(\de_i^{\al_i}\)&\ \hbox{if}\ i<j,\\
  & \\
& - 2\al_i\ln\de_i  -2\ln(1+|y|^{\al_i})+h_i(0)\\
  &\qquad+O\(\de_i|y|\)+O\(\de_i^{\al_i}\)&\ \hbox{if}\ i=j,\\
  & \\
&-2\al_i \ln\de_i+h_i(0)\\
  &\qquad+O\({ |y|^{\al_i}}\({\de_j\over\de_i}\)^{\al_i}\)+O\(\de_j|y|\)+O\(\de_i^{\al_i}\)&\ \hbox{if}\ i>j.\\
 \end{aligned}\right.\end{equation}
 Here $h_i(x):=4\pi\al_i H(x,0).$
\end{lemma}

Now, we are in position to prove the following crucial estimates.

\begin{lemma}\label{teta} Assume \eqref{alfai} and \eqref{deltai}. 
For any $j=1,\dots,k$ we have

\begin{equation}\label{er4}
\left|\Theta_j(y)\right| =O\(\delta_j|y|  +\lambda\)\quad\hbox{for any}\  {y\in {A_j\over \de_j} }
 \end{equation}
and in particular
\begin{equation}\label{er4.1}
\sup\limits_{y\in {A_j\over \de_j} }\left|\Theta_j(y)\right| =O(1). \end{equation}

\end{lemma}
\begin{proof} 
First of all, it is useful to estimate the projection $Pw_i.$

By Lemma \ref{pwi-lem}  (also using the mean value theorem $h_j \(\de_j|y|\)=h_j(0)+O\(\de_j|y|\)$), by \eqref{alfai} and by \eqref{deltai} we
deduce
\begin{align*}
  \Theta_j(y)   = &  \[ -\al_j\ln\de_j -\ln(2\al_j^2)+h_j(0)+O\(\de_j|y|\)+O\(\de_j^{\al_j}\)\]-(\al_j-2)\ln|\de_j y| \nonumber
\\   &
+\sum\limits_{i< j} (-1)^{i-j}\[-2\al_i\ln\(\de_j|y|\)
+h_i(0)+O\({1\over|y|^{\al_i}}\({\de_i\over\de_j}\)^{\al_i}\)+O\(\de_j|y|\)+O\(\de_i^{\al_i}\)\]
\nonumber
\\   &+\sum\limits_{i> j} (-1)^{i-j}\[-2\al_i \ln\de_i+h_i(0)+O\({ |y|^{\al_i}}\({\de_j\over\de_i}\)^{\al_i}\)+O\(\de_j|y|\)+O\(\de_i^{\al_i}\)\]\nonumber
\\   &+\ln\la \nonumber
\\
 = & \underbrace{ \[ -  \al_j\ln\de_j -2\sum\limits_{i> j} (-1)^{i-j} \al_i \ln\de_i-\ln(2\al_j^2)+\sum\limits_{i=1}^k(-1)^{i-j}h_i(0)+\ln\la\]}_{=\ 0\ \hbox{because of}\ \eqref{deltai}}
 \nonumber
\\   &-\underbrace{\[(\al_j-2)+2\sum\limits_{i< j} (-1)^{i-j} \al_i\]}_{=\ 0\ \hbox{because of}\ \eqref{alfai}}\ln\(\de_j|y|\)\nonumber
\\  &+O\(\de_j|y|\)+\sum\limits_{i= 1}^kO\(\de_i^{\al_i}\)+\sum\limits_{i< j}O\({1\over|y|^{\al_i}}\({\de_i\over\de_j}\)^{\al_i}\)
+\sum\limits_{i> j}  O\({
|y|^{\al_i}}\({\de_j\over\de_i}\)^{\al_i}\) \nonumber
\\   = & O\(\de_j|y|\)+\sum\limits_{i= 1}^kO\(\de_i^{\al_i}\)+\sum\limits_{i< j}O\({1\over|y|^{\al_i}}\({\de_i\over\de_j}\)^{\al_i}\)
+\sum\limits_{i> j}  O\({
|y|^{\al_i}}\({\de_j\over\de_i}\)^{\al_i}\).
\end{align*}

 By \eqref{para} we  deduce that
$$ O\(\de_i^{\al_i}\)=O\(\la^{2k-2i+1}\)=O\(\la\)\ \hbox{because}\ 1\le i\le k.$$
Moreover,   if $y\in {A_j\over \de_j}$ then $\sqrt
{\de_{j-1}\over\de_j}\le |y|\le \sqrt {\de_{j+1}\over\de_j}$
and so
 if $j=2,\dots,k$ and $i<j$ we have
\begin{align*}
 O\({1\over|y|^{\al_i}}\({\de_i\over\de_j}\)^{\al_i}\)=&O\( \({\de_i^2\over\de_{j-1}\de_j}\)^
{\al_i\over2}\)=O\( \({\de_{j-1}\over\de_j}\)^{\al_i\over2}\)=
O\(\la^{\frac{ 2k}{4(j-1)^2-1}(2i-1)}\) \\  =&O\(\la^{\frac{2k}{2k-1}}\)=O\(\la\) ,\end{align*}
(since the minimum of   $\lambda$'s exponent is achieved when $i=j-1$ and $j=k  $)
and if $j=1,\dots,k-1$ and $i>j$ we have
 \begin{align*}
O\({ |y|^{\al_i}}\({\de_j\over\de_i}\)^{\al_i}\)= &O\(
\({\de_{j+1}\de_j\over\de_i^2}\)^{\al_i\over2}\)= O\(
\({\de_j\over\de_{j+1}}\)^{\al_i\over2}\)=
O\(\la^{\frac{ 2k}{4j^2-1}(2i-1)}\)\\
=&O\(\la^{\frac{2k}{2k-3}}\)=O\(\la\), \end{align*}
 (since the minimum of   $\lambda$'s exponent is achieved when $i=j+1$ and $j=k-1 $).
 Collecting all the previous estimates, we get \eqref{er4}.

 Estimate \eqref{er4.1} follows immediately by \eqref{er4}, because if $y \in {A_j\over \de_j}$ then
 $\delta_j |y|=O(1).$
\end{proof}

In the following, we will denote by
$$\|u\|_p:=\(\int\limits_\Omega |u(x)|^pdx\)^{1\over p}\quad \hbox{and}\quad\|u\| :=\(\int\limits_\Omega |\nabla u(x)|^2dx\)^{1\over 2}$$
the usual norms in the Banach spaces $\mathrm{L}^p(\Omega)$ and $\mathrm{H}^1_0(\Omega),$ respectively.

\section{Estimate of the error term}\label{due}

In this section we will estimate the two following error terms
 
\begin{equation}\label{rla}
 \mathcal{R}_\la(x):=-\Delta W_\la(x)-\la f\(W_\la(x)\),\ x\in\Omega \end{equation}
 
\begin{equation}\label{sla}
 \mathcal{S}_\la(x):=\la f'\(W_\la(x) \)-\sum\limits_{i=1}^k 2\alpha_i^2{|x|^{\alpha_i-2}\over (\de_i^{\alpha_i}+|x|^{\alpha_i})^2},\ x\in\Omega.\end{equation}

Here $f(s):=e^s-e^{-s}.$

\begin{lemma}\label{errore} Let $ \mathcal{R}_\la$ as in \eqref{rla}.
There exists $\eps>0$ such that for any $p\in[1,1+\eps)$ we have
$$\|\mathcal{R}_\la\|_{p}= O\(\la^{2-p\over2p(2k-1)}\).$$
\end{lemma}

\begin{proof}
 First of all we observe that
\begin{equation}\label{er1}
\mathcal{R}_\la(x)=\sum\limits_{i=1}^k(-1)^i|x|^{\al_i-2}e^{w_i(x)}-\la
e^{\sum\limits_{i=1}^k(-1)^iP{w_i(x)}}+\la
e^{\sum\limits_{i=1}^k(-1)^{i+1}P{w_i(x)}}
\end{equation}
Then, using that if $x\in A_j$ then we can write
$$\mathcal{R}_\la(x)=
\left\{\begin{aligned}
&|x|^{\al_j-2}e^{w_j(x)}-\la e^{Pw_j(x)+\sum\limits_{i=1\atop i\not= j}^k(-1)^{i-j}P{w_i(x)}}+\la e^{-Pw_j(x) -\sum\limits_{i=1\atop i\not= j}^k(-1)^{i-j}P{w_i(x)}}\\   &\quad +\sum\limits_{i =1\atop i\not=j }^k (-1)^i  |x|^{\al_i-2}e^{w_i(x)}  \quad \hbox{if $j$ is even,}\\
&-|x|^{\al_j-2}e^{w_j(x)}+\la
e^{Pw_j(x)+\sum\limits_{i=1\atop i\not=
j}^k(-1)^{i-j}P{w_i(x)}}-\la e^{-Pw_j(x) -\sum\limits_{i=1\atop
i\not= j}^k(-1)^{i-j}P{w_i(x)}}\\  &\quad +\sum\limits_{i =1\atop
i\not=j }^k (-1)^i  |x|^{\al_i-2}e^{w_i(x)}  \quad \hbox{if $j$ is
odd.}\\
\end{aligned}\right.$$
we have
\begin{eqnarray}\label{er2}
&
&\int\limits_\Omega|\mathcal{R}_\la(x)|^pdx=\sum\limits_{j=1}^k\int\limits_{A_j}|R_\la(x)|^pdx
\nonumber
\\
&&\le C\sum\limits_{j=1}^k \int\limits_{A_j}\left|
|x|^{\al_j-2}e^{w_j(x)}-\la e^{Pw_j(x)+\sum\limits_{i=1\atop
i\not= j}^k(-1)^{i-j}P{w_i(x)}}\right|^pdx \nonumber
\\
&&+C\sum\limits_{j=1}^k \int\limits_{A_j}\left|\la
e^{-Pw_j(x)+\sum\limits_{i=1\atop i\not= j}
^k(-1)^{i-j+1}P{w_i(x)}}\right|^pdx \nonumber
\\
&&+C\sum\limits_{i,j=1\atop i\not=j }^k \int\limits_{A_j}\left|
|x|^{\al_i-2}e^{w_i(x)} \right|^pdx=: I_1+I_2+I_3.
\end{eqnarray}

\medskip
Let us estimate $I_1.$ For any $j=1,\dots,k$ we have
\begin{align}\label{er3b} 
&   \int\limits_{A_j}\left| |x|^{\al_j-2}e^{w_j(x)}-\la
e^{Pw_j(x)+\sum\limits_ {i=1\atop i\not=
j}^k(-1)^{i-j}P{w_i(x)}}\right|^pdx \nonumber
\\
&   = \int\limits_{A_j}  |x|^{(\al_j-2)p}e^{pw_j(x)}\left|1-
e^{Pw_j(x)-w_j(x)-(\al_j-2)\ln|x|+\sum\limits_{i\not = j}
(-1)^{i-j}P{w_i(x)} +\ln\la} \right|^pdx \nonumber
\\
& =C\de_j^{2-2p}\int\limits_{A_j\over \de_j}
\frac{|y|^{(\al_j-2)p}}{\(1+|y|^{\al_j}\)^{2p}} \left|1-
e^{Pw_j(\de_j y)-w_j(\de_j y)-(\al_j-2)\ln|\de_j
y|+\sum\limits_{i\not= j} (-1)^{i-j}P{w_i(\de_j y)} +\ln\la}
\right|^pdy \nonumber
\\ &   \hbox{(we use that $e^t-1=e^{\theta t}t$  for some
$\theta\in(0,1) $ and we   use Lemma \ref{teta})}\nonumber
\\ &
=O\(\de_j^{2-2p}\int\limits_{A_j\over \de_j}\frac{|y|^{(\al_j-2)p}}{\(1+|y|^{\al_j}\)^{2p}}\left|
 \Theta_j(y)\right|^pdy\)=\nonumber
\\ &  =O\(\de_j^{2-2p}\int\limits_{A_j\over \de_j}\frac{|y|^{(\al_j-2)p}}{\(1+|y|^{\al_j}\)^{2p}}\left|
 \delta_j |y|+\lambda\right|^pdy\)=
O\(\de_j^{2-2p}\lambda^p\)+O\(\de_j^{2- p}\)\nonumber \\  &=O\(\de_1^{2-2p}\lambda^p\)+O\(\de_k^{2- p}\)=
O\(\la^{p+(1-p)(2k-1)}\)+O\(\la^{2-p\over2(2k-1)}\)=\nonumber
\\ &  =O\(\la^{2-p\over2(2k-1)}\),
\end{align}

provided $p$ is close enough to $1.$
Therefore, we get
\begin{equation}\label{i1}
 I_1 =O\(\la^{2-p\over2(2k-1)}\).\end{equation}

 \medskip
Let us estimate $I_2$.   For any $j=1,\dots,k$,
\begin{align}\label{er5}
&  \int\limits_{A_j}\left|\la e^{-Pw_j(x)+\sum\limits_{i=1\atop i\not= j}^k(-1)^{i-j+1}P{w_i(x)}}\right|^pdx\nonumber\\
&    = \la^{2p}\de_j^2\int\limits_
{\sqrt {\de_{j-1}\over\de_j}\le |y|\le \sqrt {\de_{j+1}\over\de_j}}  e^{p(- w_j(\de_j y)-\(\al_j-2\)\ln|\de_j y| -\Theta_j(  y))} dy\nonumber\\
 & = C\la^{2p}\de_j^{2+2p}\int\limits_
{\sqrt {\de_{j-1}\over\de_j}\le |y|\le \sqrt {\de_{j+1}\over\de_j}}{\(1+|y|^{\al_j}\)^{2p}\over|y|^{(\al_j-2)p}} e^{ -p\Theta_j(  y)} dy
\nonumber\\
&  = O\( \la^{2p}\de_j^{2+2p}
\int\limits_
{\sqrt {\de_{j-1}\over\de_j}\le |y|\le \sqrt {\de_{j+1}\over\de_j}}{\(1+|y|^{\al_j}\)^{2p}\over|y|^{(\al_j-2)p}} dy\)\nonumber\\
&  \hbox{(we agree that  $\delta_0=0$ and $\delta_{k+1}=+\infty$})\nonumber\\
& = O\(  \la^{2p}\de_j^{2+2p}\[\({\de_{j+1}\over\de_j}\)^{p{\al_j+2\over2} +1}+\({\de_j\over\de_{j-1}}\)^{p{\al_j-2\over2} -1}\]\)\nonumber\\ &
 \hbox{(if $j=1$ only the first term appears, while if $j=k$ only  the second term appears})\nonumber\\
& = O\(  \la^{2p}\de_j^{2+2p}\[\({\de_{j+1}\over\de_j}\)^{2jp-1}+\({\de_j\over\de_{j-1}}\)^{(2j-2)p-1}\]\)\nonumber\\
 & \hbox{(since the best rate is obtained as }j=2)\nonumber\\
& =O\(  \la^{2p}\de_2^{2+2p}\[\({\de_3\over\de_2}\)^{4p+1}+\({\de_2\over\de_1}\)^{2p-1}\]\)\nonumber\\
&  \(\hbox{since }\frac{2k}{15}(4p+1)<\frac{2k}3(2p-1)\)\nonumber\\
& =O\(  \la^{2p+\frac{2k-3}3(1+p)-\frac{2k}3(2p-1)}\)=O\( \la^{\frac{3p-3+2k}3}\) .  \end{align}
Therefore, we get
\begin{equation}\label{i21}
 I_2  = O\( \la^{\frac{2k}3}\) .\end{equation}

 \medskip
Let us estimate $I_3$.  For any $i,j=1,\dots,k$  with $i\not=j$  we have
\begin{align}\label{er6}
  &\int\limits_{A_j}\left|{ |x|^ {\alpha_i-2}\over(\de_i^{\al_i}+|x|^{\al_i})^2}\right|^p dx\nonumber\\
  &\hbox{(we scale $x=\delta_i y$)}\nonumber\\
  &=C\de_i^{2-2p} \int\limits_{{\sqrt{\de_{j-1}\de_j}\over \de_i}\le |y|\le {\sqrt{\de_j\de_{j+1}}\over \de_i}}
  {|y|^{(\al_i-2)p}\over \(1+|y|^{\al_i}\)^{2p}} dy\nonumber \\
   &=
  \left\{
  \begin{aligned}
&  O\(\de_i^{2-2p}\({ {\sqrt{\de_j\de_{j+1}}}\over \de_i^2}\)^{(\al_i-2)p+2}\)=
O\( \de_i^{2-2p} \({\de_j\over\de_{j+1}}\)^{(2i-2)p+1} \)\\ &\qquad \hbox{if}\  j=1,\dots,k-1\ \hbox{and}\ i>j,\\
& \\
  &O\(\de_i^{2-2p}\({ \de_i\over\sqrt{\de_{j-1}\de_j} }\)^{-(\al_i+2)p+2}\)=O\(\de_i^{2-2p}\({  \de_{j-1}\over \de_j  }\)
  ^{2ip-1}\)\\ &\qquad \hbox{if}\  j=2,\dots,k \ \hbox{and}\  i<j.\\
  \end{aligned}
  \right.
  \nonumber\\
  & \nonumber\\ &
 = \left\{
  \begin{aligned}
&O\( \de_2^{2-2p} \({\de_1\over\de_{2}}\)^{2p+1} \)=O\(\la^{\frac{2kp+4k+3(p-1)}3}\)=O\(\la^  {2k }\),\\
  &O\(\de_1^{2-2p}\({  \de_{1}\over \de_2}\)
  ^{2p-1}\)=O\(\la^{\frac{-2kp+4k+3(p-1)}3}\) =O\(\la^{{2\over3}k+(1-p)\({2\over3}k-1\)}\) \\
  \end{aligned}
  \right.\nonumber\\ & \nonumber\\ &=O\(\la^{\frac{2kp+3(p-1)}3}\).
  \end{align}

Therefore, if $p$ is close enough to $1$ we get
\begin{equation}\label{i31}
 I_3  = O\( \la^{{2\over3}k+(1-p)\({2\over3}k-1\)}\) .\end{equation}

   \end{proof}

\begin{lemma}\label{tec1} Let $ \mathcal{S}_\la$ as in \eqref{sla}.
There exists $\eps>0$ such that for any $p\in[1,1+\eps)$ we have
$$\left\|\mathcal{S}_\la\right\|_{p}= O\(\la^{2-p\over{2p(2k-1)}}\).$$
\end{lemma}

\begin{proof}
By \eqref{anelli} we get
\begin{align*}
&\int\limits_\Omega\left|\mathcal{S}_\la(x)\right|^pdx =
\sum\limits_{j=1}^k\int\limits_{A_j}\left|\mathcal{S}_\la(x)\right|^pdx\\ &=O\(\sum\limits_{j=1}^k \int\limits_{A_j}\left|\la f'\(W_\la \)(x)- 2\alpha_j^2{|x|^{\alpha_j-2}\over (\de_j^{\alpha_j}+|x|^{\alpha_j})^2}\right|^pdx\)\\ &+O\(\sum\limits_{i,j=1\atop i\not=j}^k\int\limits_{A_j}\left| 2\alpha_i^2{|x|^{\alpha_i-2}\over (\de_i^{\alpha_i}+|x|^{\alpha_i})^2}\right|^pdx\):=J_1+J_2
\end{align*}
The integral $J_2$ was estimated in \eqref{er6}:
$$J_2=O\(\la^{{2\over3}k+(1-p)\({2\over3}k-1\)}\).$$

Let us estimated $J_1.$ For any $j=1,\dots,k,$ we will scale $x=\de_j y.$
We observe that by \eqref{tetaj}
 $$W_\la(\de_jy)=(-1)^j\[\Theta_j(y)+w_j(\de_jy)+(\alpha_j-2)\ln|\de_jy|-\ln \la\]$$
 and so
\begin{align*}
&\la f'\(W_\la(\de_jy)\)\\ &= \la e^{(-1)^j\[\Theta_j(y)+w_j(\de_jy)+(\alpha_j-2)\ln|\de_jy|-\ln \la\]}  +
\la e^{(-1)^{j+1}\[\Theta_j(y)+w_j(\de_jy)+(\alpha_j-2)\ln|\de_jy|-\ln \la\]}\\
&=  e^{ \[\Theta_j(y)+w_j(\de_jy)+(\alpha_j-2)\ln|\de_jy| \]} +
\la ^2 e^{-\[\Theta_j(y)+w_j(\de_jy)+(\alpha_j-2)\ln|\de_jy| \]}\\
&={2\alpha_j^2\over \de_j^2}{|y|^{\alpha_j-2}\over (1+|y|^{\alpha_j})^2}e^{\Theta_j(y)}+\la^2\de_j^2 {(1+|y|^{\alpha_j})^2\over 2\alpha_j^2  |y|^{\alpha_j-2} }e^{-\Theta_j(y)},
\end{align*}
from which we deduce taking also into account Lemma \ref{teta}
\begin{align*}
&\int\limits_{A_j}\left|\la f'\(W_\la \)(x)- 2\alpha_j^2{|x|^{\alpha_j-2}\over (\de_j^{\alpha_j}+|x|^{\alpha_j})^2}\right|^pdx\\ &=
 \de_j^2\int\limits_{A_j\over \de_j}\left|\la  f'\(W_\la(\de_jy)\)-{2\alpha_j^2 \over \de_j^2 }{|y|^{\alpha_j-2}\over (1+|y|^{\alpha_j})^2}\right|^pdy\\ &=
 O\(\de_j^{2-2p}\int\limits_{A_j\over \de_j}\left|{2\alpha_j^2 }{|y|^{\alpha_j-2}\over (1+|y|^{\alpha_j})^2}\(e^{\Theta_j(y)}-1\)\right|^pdy\)\\ & +O\(\la^{2p}\de_j^{2+2p} \int\limits_{A_j\over \de_j} \left|{(1+|y|^{\alpha_j})^2\over 2\alpha_j^2  |y|^{\alpha_j-2} }e^{-\Theta_j(y)}\right|^pdy\)\\
&=O\(\de_j^{2-2p}\int\limits_{A_j\over \de_j}\left| {|y|^{\alpha_j-2}\over (1+|y|^{\alpha_j})^2} \Theta_j(y) \right|^pdy\)\\ & +O\(\la^{2p}\de_j^{2+2p} \int\limits_{A_j\over \de_j} \left|{(1+|y|^{\alpha_j})^2\over   |y|^{\alpha_j-2} } \right|^pdy\)\\ &\hbox{(the first term is estimated in \eqref{er3b} and the second term is estimated in \eqref{er5})}\\ &
=O\(\la^{2-p  \over 2(2k-1)}\)+O\(\la^{{2\over3}k }\).
\end{align*}
 Therefore, we get
 $$J_1=O\(\la^{2-p  \over 2(2k-1)}\)+O\(\la^{{2\over3}k }\).$$

Finally, the claim follows collecting all the previous estimates.

\end{proof}
\section{The linear theory}\label{tre}

Let us consider the linear operator
\begin{equation}\label{lla}\mathcal{L}_{\la}(\phi ):=-\Delta \phi  - \(\sum\limits_{i=1}^k 2\alpha_i^2{|x|^{\alpha_i-2}\over (\de_i^{\alpha_i}+|x|^{\alpha_i})^2}\)\phi.\end{equation}

Let us   study  the invertibility of  the linearized operator $\mathcal{L}_{\la}.$

\begin{prop}\label{inv}
For any $p>1$ there exists $\la_0>0$ and $c>0$ such that for any $\la \in(0, \la_0)$ and for any $h\in \mathrm{L}^{p}(\Omega)$ there exists a unique
$\phi\in \mathrm{W}^{2, 2}(\Omega)$ solution of
$$ \mathcal{L}_{\la}(\phi )=\psi\ \hbox{in}\ \Omega,\ \phi=0\ \hbox{on}\ \partial\Omega,
$$
which satisfies $$\|\phi\| \leq c |\ln\la|  \|h\|_{p}.$$
\end{prop}
\begin{proof}
We argue by contradiction. Assume there exist $p>1,$ sequences $\la_n\to0,$ $\psi_n\in \mathrm{L}^{\infty}(\Omega)$ and $\phi_n\in \mathrm{W}^{2, 2}(\Omega)$
such that
\begin{equation}\label{inv1}
-\Delta \phi_n  -\sum\limits_{i=1}^k2\alpha_i^2{ {\de_i}_n ^{\alpha_i}|x|^{\alpha_i-2}\over ({\de_i}_n ^{\alpha_i}+|x|^{\alpha_i})^2}\phi_n=\psi_n,\ \hbox{in}\ \Omega,\ \phi_n=0\ \hbox{on}\ \partial\Omega,
\end{equation}
with ${\de_1}_n ,\dots,{\de_k}_n  $   defined as in  \eqref{para} and
\begin{equation}\label{inv2}
\|\phi_n\| =1\quad\hbox{and}\quad    |\ln\la_n|  \|\psi_n\|_p\to0.\end{equation}

For any $j=1,\dots,k$ we define
 $  \phi^j_n(y):=\phi_n\({\de_i}_n y\) $ with $y\in \Omega^j_n:={\Omega\over \de_n^j}.$

For sake of simplicity, in the following we will omit the index $n$ in all the sequences.

{\em Step 1: we will show that
\begin{equation}\label{step1}  \phi^j (y)\to \gamma_j {1-|y|^{\al_j}\over1+|y|^{\al_j}}\ \hbox{  for some $ \gamma_j\in\rr.$} \end{equation}
weakly in $\mathrm{H}_{\al_j}(\rr^2)$ and strongly in $\mathrm{L}_{\al_j}(\rr^2) $ (see \eqref{hjs} and \eqref{hjs}).}

First of all we claim that each $\phi^j$ is bounded in the space $\mathrm{H}_{\al_j}(\rr^2)$
defined in \eqref{hjs}.

Indeed,
  if we multiply  \eqref{inv1} by $\phi$ we deduce that for any $j$
 \begin{align*}  \int\limits_{\Omega }2\al_j^2 {\de_j^{\al_j}|x|^{\al_j-2}\over(\de_j^{\al_j}+|x|^{\al_j})^2} \phi^2 (x)dx&\le \sum\limits_{i=1}^k\int\limits_{\Omega }2\al_i^2 {\de_i^{\al_i}|x|^{\al_i-2}\over(\de_i^{\al_i}+|x|^{\al_i})^2} \phi^2 (x)dx\\ &=\int\limits_{\Omega }|\nabla\phi (x)|^2dx-\int\limits_{\Omega }\psi(x)\phi(x)dx\\ &=1+O\(\|\psi\|_p\|\phi\|\) =O(1)\end{align*}
Our claim follows since by scaling
$$\int\limits_{\Omega^j}|\nabla\phi^j(y)|^2dy=\de^2_j\int\limits_{\Omega^j}|\nabla\phi (\de_j y)|^2dy=\int\limits_{\Omega }|\nabla\phi (x)|^2dx=1.$$
and
 $$\int\limits_{\Omega^j}2\al_j^2 {|y|^{\al_j-2}\over(1+|y|^{\al_j})^2}\(\phi^j(y)\)^2dy=\int\limits_{\Omega }2\al_j^2 {\de_j^{\al_j}|x|^{\al_j-2}\over(\de_j^{\al_j}+|x|^{\al_j})^2} \phi^2 (x)dx .$$

Therefore,  by Proposition \eqref{compact} we can assume that (up to a subsequence)
  $\phi^j \rightharpoonup \phi^j_0$ weakly in $\mathrm{H}_{\al_j}(\rr^2)$ and strongly in $\mathrm{L}_{\al_j}(\rr^2).$

Now, we point out that 
each function $  \phi^j $ solves the problem
\begin{equation}\label{s2.3}
 -\Delta   \phi^j  =2\alpha_j^2{|y|^{\alpha_j-2}\over (1+|y|^{\alpha_j})^2}  \phi^j+\rho_j(y)\phi^j+\de_j^2\psi(\de_j y)\ \hbox{in}\ \Omega^j ,\ \phi^j =0\ \hbox{on}\ \partial\Omega^j ,
\end{equation}
where
\begin{equation}\label{s2.4}
\rho^j(y):=\sum\limits_{i=1\atop i\not=j}^k2\alpha_i^2{ {\de_i}  ^{\alpha_i}{\de_j}  ^{\alpha_j}|y|^{\alpha_i-2}\over ({\de_i}  ^{\alpha_i}+{\de_j}  ^{\alpha_i}|y|^{\alpha_i})^2}.\end{equation}

Now, let $\varphi\in C^\infty_0(\rr^2)$ be a given function and let $\mathcal{K}$ its support. It is clear that if $n$ is large enough
$$\mathcal{K}\subset {A_j\over\de_j}=\left\{y\in \Omega^j\ :\ \sqrt{\de_{j-1}\over\de_j}\le |y|\le \sqrt{\de_{j+1}\over\de_j}\right\},$$
where $A_j$ is the annulus defined in \eqref{anelli}.
We multiply equation \eqref{s2.3} by $\varphi$ and we get
\begin{align*}
&\int\limits_{\mathcal{K}}\nabla\phi^j(y)\nabla\varphi(y) dy-\int\limits_{\mathcal{K}}2\alpha_j^2{|y|^{\alpha_j-2}\over (1+|y|^{\alpha_j})^2}  \phi^j(y)\varphi(y)  dy\\
&=
\sum\limits_{i=1\atop i\not=j}^k\int\limits_{\mathcal{K}} 2\alpha_i^2{ {\de_i}  ^{\alpha_i}{\de_j}  ^{\alpha_j}|y|^{\alpha_i-2}\over ({\de_i}  ^{\alpha_i}+{\de_j}  ^{\alpha_j}|y|^{\alpha_i})^2} \phi^j(y)\varphi(y)  dy+\int\limits_{\mathcal{K}}\de_j^2\psi(\de_j y)\varphi(y)  dy.\end{align*}
Therefore, passing to the limit we get
\begin{align}\label{phi0}
&\int\limits_{\mathcal{K}}\nabla\phi^j_0(y)\nabla\varphi(y) dy-\int\limits_{\mathcal{K}}2\alpha_j^2{|y|^{\alpha_j-2}\over (1+|y|^{\alpha_j})^2}  \phi^j_0(y)\varphi(y)  dy=0\ \forall\ \varphi\in C^\infty_0(\rr^2),\end{align}
because
\begin{align*}
&
\sum\limits_{i=1\atop i\not=j}^k\int\limits_{\mathcal{K}} 2\alpha_i^2{ {\de_i}  ^{\alpha_i}{\de_j}  ^{\alpha_j}|y|^{\alpha_i-2}\over ({\de_i}  ^{\alpha_i}+{\de_j}  ^{\alpha_j}|y|^{\alpha_i})^2} \phi^j(y)\varphi(y)  dy\\ &=O\(\sum\limits_{i=1\atop i\not=j}^k\int\limits_{{A_j\over\de_j}} 2\alpha_i^2{ {\de_i}  ^{\alpha_i}{\de_j}  ^{\alpha_j}|y|^{\alpha_i-2}\over ({\de_i}  ^{\alpha_i}+{\de_j}  ^{\alpha_j}|y|^{\alpha_i})^2} |\phi^j(y)|  dy\)\ \hbox{(because $\mathcal{K}\subset {A_j\over\de_j}$)}\\
&=O\(\sum\limits_{i=1\atop i\not=j}^k \int\limits_{{A_j }} 2\alpha_i^2{ {\de_i}  ^{\alpha_i} |x|^{\alpha_i-2}\over ({\de_i}  ^{\alpha_i}+ |x|^{\alpha_i})^2} |\phi (x) |dx\)\ \hbox{(we scale $ x=\de_j y$)}\\
&=O\(\sum\limits_{i=1\atop i\not=j}^k\(\int\limits_{{A_j }}\left| 2\alpha_i^2{ {\de_i}  ^{\alpha_i} |x|^{\alpha_i-2}\over ({\de_i}  ^{\alpha_i}+ |x|^{\alpha_i})^2} dx \right|^p\)^{1/p} \|\phi \|_q\)\ \hbox{(we use H\"older's estimate)}\\ &=o(1) \ \hbox{(we use   estimate \eqref{er6} and the fact that  $|\phi |_q\le 1$)}   \end{align*}
and
$$\int\limits_{\mathcal{K}}\de_j^2\psi(\de_j y)\varphi(y)  dy=O\(\int\limits_{\Omega^j}\de_j^2|\psi(\de_j y)|  dy\)=O\(\int\limits_{\Omega } |\psi(x)|  dx\)=O(\|\psi\|_p)=o(1).$$
By \eqref{phi0} we deduce that $\phi^j_0$ is a solution to the equation
$$-\Delta \phi^j_0 =2\alpha_j^2{|y|^{\alpha_j-2}\over (1+|y|^{\alpha_j})^2}\phi^j_0\ \hbox{in}\ \rr^2\setminus\{0\}.$$
Finally, since $\int\limits_{\rr^2}|\nabla \phi^j_0(y)|^2dy\le1$ it is standard to see that $\phi^j_0$ is a solution in the whole space $\rr^2.$
By Theorem \ref{esposito}  we get the claim.

\medskip
{\em Step 2: we will show that $\gamma_j=0$ for any $j=1,\dots,k.$}

Here we are  inspired by some ideas used by Gladiali-Grossi \cite{gg}.

We set
\begin{equation}\label{sigmai}\sigma_i(\lambda):= \ln\la \int\limits_{\Omega^i} 2\al_i^2{ |y|^{\al_i-2}\over \(1+|y|^{\al_ i}\)^2}\phi^i(y)dy.\end{equation}
We will show that
\begin{equation}\label{sigma}
\sigma_i :=\lim\limits_{\la\to0} \sigma_i(\lambda)=0\ \hbox{for any}\ i=1,\dots,k.
\end{equation}
We know that $\phi$ solves the problem (see \eqref{s2.3})
\begin{equation}\label{cru0}
-\Delta  \phi=\sum\limits_{j=1}^k2\al_j^2{\de_j^{\al_j}|x|^{\al_j-2}\over \(\de_j^{\al_j}+|x|^{\al_j}\)^2}\phi+\psi\ \hbox{in}\ \Omega,\ \phi=0\ \hbox{on}\ \partial\Omega.\end{equation}
Set $Z_i(x):={\de_i^{\al_i}-|x|^{\al_i}\over \de_i^{\al_i}+|x|^{\al_i}.}$ We know that 
$Z_i$ solves (see Theorem \ref{esposito})
$$-\Delta Z_i=2\al_i^2{\de_i^{\al_i}|x|^{\al_i-2}\over \(\de_i^{\al_i}+|x|^{\al_i}\)^2}Z_i\quad\hbox{in}\ \rr^2.$$
Let $PZ_i$ be its projection onto $\mathrm{H}^1_0(\Omega) $ (see \eqref{pro}), i.e.
\begin{equation}\label{cru1}
-\Delta  PZ_i=2\al_i^2{\de_i^{\al_i}|x|^{\al_i-2}\over \(\de_i^{\al_i}+|x|^{\al_i}\)^2}Z_i\ \hbox{in}\ \Omega,\ PZ_i=0\ \hbox{on}\ \partial\Omega.\end{equation}
By maximum principle (see also Lemma \ref{pwi-lem}) we deduce that
\begin{equation}\label{pz}PZ_i(x)=Z_i(x)+1+O\(\de_i^{\al_i}\)={2\de_i^{\al_i} \over \de_i^{\al_i}+|x|^{\al_i} }+O\(\de_i^{\al_i}\)\end{equation}
frow which we get
\begin{equation}\label{cru2}
  PZ_i(\de_j y)=\left\{\begin{aligned}
  & O\({1\over|y|^{\al_i}}\({\de_i\over\de_j}\)^{\al_i}\) +O\(\de_i^{\al_i}\)\ \hbox{if}\ i<j,\\
&{2  \over 1+|y|^{\al_i} }+ O\(\de_i^{\al_i}\)\ \hbox{if}\ i=j,\\
& 2 +O\({ |y|^{\al_i}}\({\de_j\over\de_i}\)^{\al_i}\)+ O\(\de_i^{\al_i}\)\ \hbox{if}\ i>j.\\
 \end{aligned}\right.\end{equation}

Now,  we multiply \eqref{cru0} by $(\ln\la)  PZ_i$ and \eqref{cru1} by $(\ln\la)\phi.$ If we subtract the two equations obtained, we get
\begin{align*}
\ln\la \int\limits_\Omega   2\al_i^2{\de_i^{\al_i}|x|^{\al_i-2}\over \(\de_i^{\al_i}+|x|^{\al_i}\)^2}\phi(x) Z_i(x)dx=&
 \ln\la\sum\limits_{j=1 }^k\int\limits_\Omega   2\al_j^2{\de_j^{\al_j}|x|^{\al_j-2}\over \(\de_j^{\al_j}+|x|^{\al_j}\)^2}\phi(x)PZ_i(x)dx\\ &+\ln\la\int\limits_\Omega \psi(x) PZ_i(x)dx   \end{align*}
and so
\begin{align}\label{cru3}
 &\ln\la\int\limits_\Omega   2\al_i^2{\de_j^{\al_i}|x|^{\al_i-2}\over \(\de_i^{\al_i}+|x|^{\al_i}\)^2}\phi(x) \(PZ_i(x)-Z_i(x)\) dx\nonumber\\ &+
 \ln\la\sum\limits_{j=1\atop j\not=i }^k\int\limits_\Omega   2\al_j^2{\de_j^{\al_j}|x|^{\al_j-2}\over \(\de_j^{\al_j}+|x|^{\al_j}\)^2}\phi(x)PZ_i(x)dx\nonumber\\ &+\ln\la\int\limits_\Omega \psi(x) PZ_i(x)dx=0.   \end{align}

We are going to pass to the limit in \eqref{cru3}.

The last term is  
\begin{equation}\label{cru4}\ln\la\int\limits_\Omega \psi(x) PZ_i(x)dx=O\(|\ln\la|\|\psi\|_p\)=o(1),\end{equation}
because of \eqref{inv2} and since by \eqref{pz} we get
$\|PZ_i\|_\infty=O(1).$

The first term is 
\begin{align}\label{cru5}
 &\ln\la\int\limits_\Omega   2\al_i^2{\de_i^{\al_i}|x|^{\al_i-2}\over \(\de_i^{\al_i}+|x|^{\al_i}\)^2}\phi(x) \(PZ_i(x)- Z_i(x)\)dx \nonumber\\ &\qquad \hbox{(we scale $x=\de_iy$ and we apply \eqref{pz})}\nonumber\\
&=\ln\la \int\limits_{\Omega^i} 2\al_i^2{ |y|^{\al_i-2}\over \(1+|y|^{\al_i}\)^2}\phi^i(y) dy +O\(\de_i^{\al_i}|\ln\la| \int\limits_{\Omega^i} 2\al_i^2{ |y|^{\al_i-2}\over \(1+|y|^{\al_i}\)^2}|\phi^i(y)| dy\)\nonumber\\ &\qquad  \hbox{(we use \eqref{sigmai} and \eqref{step1})}\nonumber\\
&= \sigma_i( \la) +o(1).\end{align}

We estimate the second term.
If $j\not=i$ we get
\begin{align}\label{cru6}
&\ln\la \int\limits_\Omega 2\al_j^2{\de_j^{\al_j}|x|^{\al_j-2}\over \(\de_j^{\al_j}+|x|^{\al_j}\)^2}\phi(x)PZ_i(x)dx\ \hbox{(we scale $x=\de_jy$)}\nonumber\\
&=\ln\la \int\limits_{\Omega^j} 2\al_j^2{ |y|^{\al_j-2}\over \(1+|y|^{\al_j}\)^2}\phi^j(y)PZ_i(\de_j y)dy\ \hbox{(we use \eqref{cru2})}\nonumber\\
&=\left\{\begin{aligned}
& 2\ln\la\int\limits_{\Omega^j} 2\al_j^2{ |y|^{\al_j-2}\over \(1+|y|^{\al_j}\)^2}\phi^j(y)dy   +\\
&\qquad +O\(|\ln\la|\int\limits_{\Omega^j} 2\al_j^2{ |y|^{\al_j-2}\over \(1+|y|^{\al_j}\)^2}|\phi^j(y)| \({ |y|^{\al_i}}\({\de_j\over\de_i}\)^{\al_i}+\de_i^{\al_i}\)dy\) \ & \hbox{if $j<i$}\\
&\\
&O\( |\ln\la|\int\limits_{\Omega^j} 2\al_j^2{ |y|^{\al_j-2}\over \(1+|y|^{\al_j}\)^2}|\phi^j(y)| \({1\over|y|^{\al_i}}\({\de_i\over\de_j}\)^{\al_i} +\de_i^{\al_i}\)dy\)\ & \hbox{if $j>i.$}\\
\end{aligned}\right.\nonumber\\
&\qquad\hbox{(we use \eqref{sigmai}, \eqref{crux}, \eqref{cruy} and  \eqref{cruz})}\nonumber\\
&=\left\{\begin{aligned}
&2\sigma_j(\la) +o(1)\ & \hbox{if $j<i$}\\
& o(1)\ & \hbox{if $j>i.$}\\
\end{aligned}\right.
\end{align}

 By \eqref{cru3}, \eqref{cru4}, \eqref{cru5} and  \eqref{cru6} we get
 $$\sigma_1(\la)=o(1)\ \hbox{and}\ \sigma_i(\la)+2\sum\limits_{j=1}^{i-1}\sigma_j(\la)=o(1)\ \hbox{for any}\ i=2,\dots,k, $$
 which implies passing to the limit and using the definition of $\sigma_i$ given in  \eqref{sigma}, 
$$\sigma_1 =0\ \hbox{and}\ \sigma_i +2\sum\limits_{j=1}^{i-1}\sigma_j =0\ \hbox{for any}\ i=2,\dots,k. $$
Therefore, \eqref{sigma} immediately follows.

We used the following three estimates.
If $j<i$ we have
\begin{align}\label{crux}
&\( |\ln\la|{\de_j\over\de_i}\)^{\al_i}\int\limits_{\Omega^j}{|y|^{\al_j+\al_i-2}\over\(1+|y|^{\al_j}\)^2}|\phi^j(y)|dy\ \hbox{(by H\"older's inequality)}\nonumber\\
&=O\( |\ln\la|\({\de_j\over\de_i}\)^{\al_i}\de_j^{2(1-p)\over p}\|\phi\|\(\int\limits_{\rr^2 }\({|y|^{\al_j+\al_i-2}\over\(1+|y|^{\al_j}\)^2} \)^pdy  \)^{1/p}\)\nonumber\\ &\qquad \hbox{(we use $\al_j >\al_i$ and we choose $p$ close to 1)}\nonumber\\
&=O\( |\ln\la|\({\de_j\over\de_i}\)^{\al_i}\de_j^{2(1-p)\over p} \)=o(1)
\end{align}
and if $j>i$ we have
\begin{align}\label{cruy}
&\( |\ln\la|{\de_i\over\de_j}\)^{\al_i}\int\limits_{\Omega^j}{1\over|y|^{\al_i-\al_j+2}\(1+|y|^{\al_j}\)^2}|\phi^j(y)|dy\ \hbox{(by H\"older's inequality)}\nonumber\\
&=O\( |\ln\la|\({\de_j\over\de_i}\)^{\al_i}\de_j^{2(1-p)\over p}\|\phi\|\(\int\limits_{\rr^2 }\({1\over|y|^{\al_i-\al_j+2}\(1+|y|^{\al_j}\)^2} \)^pdy  \)^{1/p}\)\nonumber\\ &\qquad \hbox{(we use $\al_i >\al_j$ and we choose $p$ close to 1)}\nonumber\\
&=O\( |\ln\la|\({\de_j\over\de_i}\)^{\al_i}\de_j^{2(1-p)\over p} \)=o(1);
\end{align}
moreover for any $i$ and $j$ we have
\begin{align}\label{cruz}
&  |\ln\la| \de_i ^{\al_i}\int\limits_{\Omega^j}{|y|^{\al_ j-2}\over\(1+|y|^{\al_j}\)^2}|\phi^j(y)|dy\ \hbox{(by H\"older's inequality)}\nonumber\\
&=O\( |\ln\la|\de_i ^{\al_i}\de_j^{2(1-p)\over p}\|\phi\|\(\int\limits_{\rr^2 }\({|y|^{\al_ j-2}\over \(1+|y|^{\al_j}\)^2} \)^p dy \)^{1/p}\)\nonumber\\ &\qquad \hbox{(  we choose $p$ close to 1)}\nonumber\\
&=O\( |\ln\la|\de_i ^{\al_i}\de_j^{2(1-p)\over p} \)=o(1).
\end{align}

 \medskip

Finally, we have all the ingredients to show that
\begin{equation}\label{ai}
\gamma_i  =0\ \hbox{for any}\ i=1,\dots,k.
\end{equation}

  We know that $Pw_i$  solves the problem
\begin{equation}\label{cr1}
-\Delta  Pw_i=2\al_i^2{\de_i^{\al_i}|x|^{\al_i-2}\over \(\de_i^{\al_i}+|x|^{\al_i}\)^2} \ \hbox{in}\ \Omega,\ Pw_i=0\ \hbox{on}\ \partial\Omega.\end{equation}
Now,  we multiply \eqref{cru0} by $ Pw_i$ and \eqref{cr1} by $ \phi.$ If we subtract the two equations obtained, we get
\begin{align}\label{cr2}
 \int\limits_\Omega   2\al_i^2{\de_i^{\al_i}|x|^{\al_i-2}\over \(\de_i^{\al_i}+|x|^{\al_i}\)^2}\phi(x)  dx=&
  \sum\limits_{j=1 }^k\int\limits_\Omega   2\al_j^2{\de_j^{\al_j}|x|^{\al_j-2}\over \(\de_j^{\al_j}+|x|^{\al_j}\)^2}\phi(x)Pw_i(x)dx\nonumber\\  &+ \int\limits_\Omega \psi(x) Pw_i(x)dx.  \end{align}

We want to pass to the limit in \eqref{cr2}.

The L.H.S. of \eqref{cr2} reduces to
 \begin{align}\label{cr4}
& \int\limits_\Omega   2\al_i^2{\de_i^{\al_i}|x|^{\al_i-2}\over \(\de_i^{\al_i}+|x|^{\al_i}\)^2}\phi(x)  dx\ \hbox{(we scale $x=\de_iy$)}\nonumber\\
&= \int\limits_{\Omega^i} 2\al_i^2{ |y|^{\al_i-2}\over \(1+|y|^{\al_i}\)^2}\phi^i(y) dy =o(1)\  \hbox{(because of \eqref{ex1} and \eqref{step1}).}
  \end{align}

The last term of the R.H.S. of \eqref{cr2} gives
 \begin{align}\label{cr41}\int\limits_\Omega \psi(x) Pw_i(x)dx=O\(|\ln\la|\|\psi\|_p\)o(1),\end{align}
  because of \eqref{inv2} and since by   \eqref{pro-exp}   we get
$\|Pw_i\|_\infty= O(|\ln\la|).$

Finally, we claim that  the first term of the R.H.S. of \eqref{cr2} is
  \begin{align}\label{cr42}
&  \sum\limits_{j=1 }^k\int\limits_\Omega   2\al_j^2{\de_j^{\al_j}|x|^{\al_j-2}\over \(\de_j^{\al_j}+|x|^{\al_j}\)^2}\phi(x)Pw_i(x)dx\nonumber \\ &=\left\{\begin{aligned}
  &4\pi\al_i\(\gamma_i+2\sum\limits_{j=i+1}^k \gamma_j\)+o(1)&\ \hbox{if}\ i=1,\dots,k-1,\\
  & 4\pi\al_k\gamma_k +o(1)&\ \hbox{if}\ i=1,\dots,k.\\
  \end{aligned}\right.\end{align}
  Therefore, passing to the limit, by \eqref{cr2}, \eqref{cr4}, \eqref{cr41} and \eqref{cr42}
  we immediately get
  $$ \gamma_k =0\ \hbox{and}\ \gamma_i+2\sum\limits_{j=i+1}^k \gamma_j=0\ \hbox{for any}\ i=1,\dots,k-1,$$
  which implies \eqref{ai}.

It only remains to    prove \eqref{cr42}. We have
\begin{align*}
&\int\limits_\Omega   2\al_j^2{\de_j^{\al_j}|x|^{\al_j-2}\over \(\de_j^{\al_j}+|x|^{\al_j}\)^2}\phi(x)Pw_i(x)dx\ \hbox{(we scale $x=\de_jy$)}\nonumber\\
&=\int\limits_{\Omega^j}   2\al_j^2{ |y|^{\al_j-2}\over \(1+|y|^{\al_j}\)^2}\phi^j(y)Pw_i(\de_jy)dy\ \hbox{(we use \eqref{pwi})}\nonumber\\ \end{align*}
\begin{align*}
&=\left\{\begin{aligned}
&  \int\limits_{\Omega^j} 2\al_j^2{ |y|^{\al_j-2}\over \(1+|y|^{\al_j}\)^2}\phi^j(y)\(-2\al_i \ln\de_i+h_i(0)\)dy   +\\
&\qquad +O\( \int\limits_{\Omega^j} 2\al_j^2{ |y|^{\al_j-2}\over \(1+|y|^{\al_j}\)^2}|\phi^j(y)|
 \({ |y|^{\al_i}}\({\de_j\over\de_i}\)^{\al_i}+\de_j|y|+\de_i^{\al_i}\)dy\)
\ & \hbox{if $j<i$}\\
&\\
 &  \int\limits_{\Omega^i} 2\al_i^2{ |y|^{\al_i-2}\over \(1+|y|^{\al_i}\)^2}\phi^i(y)\(-2\al_i\ln \de_i-2 \ln (1+|y|^{\al_i})   +h_i(0)\)dy   +\\
&\qquad +O\( \int\limits_{\Omega^i} 2\al_i^2{ |y|^{\al_i-2}\over \(1+|y|^{\al_i}\)^2}|\phi^i(y)|
 \( \de_i|y|+\de_i^{\al_i}\)dy\)
\ & \hbox{if $j=i$}\\
&\\
 &  \int\limits_{\Omega^j} 2\al_j^2{ |y|^{\al_j-2}\over \(1+|y|^{\al_j}\)^2}\phi^j(y)\(-2\al_i\ln\(\de_j|y|\) +h_i(0)\)dy   +\\
&\qquad +O\( \int\limits_{\Omega^j} 2\al_j^2{ |y|^{\al_j-2}\over \(1+|y|^{\al_j}\) }|\phi^j(y)|
 \({1\over|y|^{\al_i}}\({\de_i\over\de_j}\)^{\al_i}+\de_j|y|+\de_i^{\al_i}\)dy\)
\ & \hbox{if $j>i$}\\
&\\
\end{aligned}\right.\nonumber\\
& \hbox{(we use  the relation between $\de_i$ and $\la$ in \eqref{para} and we use \eqref{crux}, \eqref{cruy}, \eqref{cruz}  and \eqref{cruw})}\\\end{align*}
\begin{align*}
 &=\left\{\begin{aligned}
&  \int\limits_{\Omega^j} 2\al_j^2{ |y|^{\al_j-2}\over \(1+|y|^{\al_j}\)^2}\phi^j(y) \[-2\al_i \ln d_i-2\(2(k-i)+1\)\ln\la+h_i(0)\] dy   \\
&\qquad +o(1) \  \hbox{if $j<i$}\\
&\\
 &  \int\limits_{\Omega^i} 2\al_i^2{ |y|^{\al_i-2}\over \(1+|y|^{\al_i}\)^2}\phi^j(y)\[-2\al_i \ln d_i-2\(2(k-i)+1\)\ln\la-2 \ln (1+|y|^{\al_i}) +h_i(0) \]dy  \\
&\qquad  +o(1) \  \hbox{if $j=i$}\\
&\\
 &  \int\limits_{\Omega^j} 2\al_j^2{ |y|^{\al_j-2}\over \(1+|y|^{\al_j}\)^2}\phi^j(y)\[-2\al_i \ln d_j-2\(2(k-j)+1\)\ln\la -2\al_i\ln|y| +h_i(0) \] dy  \\
&\qquad   + o(1) \  \hbox{if $j>i$} \\
\end{aligned}\right.\nonumber\\
& \hbox{(we use  the definition of $\sigma_i$ in   \eqref{sigmai} and  we use \eqref{step1} and  \eqref{ex1})}\\
 \end{align*}
 \begin{align*}
 &=\left\{\begin{aligned} &-2\(2(k-i)+1\)\sigma_j(\la) +o(1)
 \\
&\qquad\qquad \hbox{if $j<i$}\\
&\\
 & -2\(2(k-i)+1\)\sigma_i(\la)+ \int\limits_{\Omega^i} 2\al_i^2{ |y|^{\al_i-2}\over \(1+|y|^{\al_i}\)^2}\phi^i(y)\[ -2 \ln (1+|y|^{\al_i})   \]dy    +o(1)  \\
&\qquad \qquad  \hbox{if $j=i$}\\
&\\
 &  -2\(2(k-j)+1\)\sigma_j(\la)+\int\limits_{\Omega^j} 2\al_j^2{ |y|^{\al_j-2}\over \(1+|y|^{\al_j}\)^2}\phi^j(y)\[  -2\al_i\ln|y|   \] dy     + o(1)  \\
&\qquad \qquad \hbox{if $j>i$} \\
\end{aligned}\right.\nonumber\\
& \hbox{(we use     \eqref{sigma}  and  \eqref{step1}    because $\ln (1+|y|^{\al_j}),\ln|y|\in \mathrm{L}_{\al_j}(\rr^2)$ )}\\
 \end{align*}
   \begin{align*} 
   &=\left\{\begin{aligned} &o(1)
 &\  \hbox{if $j<i$}\\
&\\
 &  \gamma_i \int\limits_{\rr^2} 2\al_i^2{ |y|^{\al_i-2}\over \(1+|y|^{\al_i}\)^2}{ 1-|y|^{\al_i }\over  1+|y|^{\al_i}  } \[ -2 \ln (1+|y|^{\al_i})   \]dy    +o(1)   &\  \hbox{if $j=i$}\\
&\\
 &    \gamma_j \int\limits_{\rr^2} 2\al_j^2{ |y|^{\al_j-2}\over \(1+|y|^{\al_j}\)^2}{ 1-|y|^{\al_j }\over  1+|y|^{\al_j}  }\[  -2\al_i\ln|y|   \] dy     + o(1)  &\ \hbox{if $j>i$} \\
\end{aligned}\right.\nonumber\\
& \hbox{(we use    \eqref{ex2} and \eqref{ex3})}\\
 \end{align*}
   \begin{align}\label{cr5} &=\left\{\begin{aligned}
&  o(1) \ & \hbox{if $j<i$}\\
&\\
 &  4\pi\al_i \gamma_i  +o(1) \ & \hbox{if $j=i$}\\
&\\
 &  8\pi\al_i  \gamma_j    + o(1) \ & \hbox{if $j>i$} \\
\end{aligned}\right.
 \end{align}

If we sum \eqref{cr5} over the index $j$ we get \eqref{cr42}.

We used the following estimate.
For any   $j$ we have
\begin{align}\label{cruw}
&    \de_j  \int\limits_{\Omega^j}{|y|^{\al_ j-1 }\over\(1+|y|^{\al_j}\)^2}|\phi^j(y)|dy\ \hbox{(by H\"older's inequality)}\nonumber\\
&=O\( \de_j  \de_j^{2(1-p)\over p}\|\phi\|\(\int\limits_{\rr^2 }\({|y|^{\al_ j-1}\over \(1+|y|^{\al_j}\)^2} \)^p dy \)^{1/p}\)\nonumber\\ &\qquad \hbox{(  we choose $p$ close to 1)}\nonumber\\
&=O\(   \de_j^{2-p\over p} \)=o(1).
\end{align}

  A straightforward computation leads to
  \begin{align}\label{ex1}
 &\int\limits_\Omega   2\al_i^2{ |y|^{\al_i-2}\over \(1+|y|^{\al_i}\)^2}{  1-|y|^{\al_i} \over 1+|y|^{\al_i} }dy=0,\\
\label{ex2}
 &\int\limits_\Omega   2\al_i^2{ |y|^{\al_i-2}\over \(1+|y|^{\al_i}\)^2}{  1-|y|^{\al_i} \over 1+|y|^{\al_i} }\ln\(1+|y|^{\al_i}\)^2dy=-4\pi\al_i,\\
\label{ex3}
 &\int\limits_\Omega   2\al_i^2{ |y|^{\al_i-2}\over \(1+|y|^{\al_i}\)^2}{  1-|y|^{\al_i} \over 1+|y|^{\al_i} }\ln|y|dy=-4\pi.
 \end{align}

\medskip
{\em Step 3: we will show that a contradiction arises!}
 We multiply equation \eqref{inv1} by $\phi$  and we get
\begin{align*}
1 &=\sum\limits_{i=1}^k\int\limits_{\Omega}2\alpha_i^2{ {\de_i}  ^{\alpha_i}|x|^{\alpha_i-2}\over ({\de_i}  ^{\alpha_i}+|x|^{\alpha_i})^2}\phi^2(x)dx+\int\limits_{\Omega}\psi(x)\phi(x)dx    \\
&=\sum\limits_{i=1}^k \int\limits_{\Omega^i} 2\alpha_i^2{  |y|^{\alpha_i-2}\over (1+|y|^{\alpha_i})^2}\(\phi^i(y)\)^2dy +O\(\|\psi\|_p\|\phi\|\) \  \hbox{(we use \eqref{inv2})}\\
&=\sum\limits_{i=1}^k \int\limits_{\Omega^i} 2\alpha_i^2{  |y|^{\alpha_i-2}\over (1+|y|^{\alpha_i})^2}\(\phi^i(y)\)^2dy +o(1)\\
 &=o(1)\ \hbox{(because  $\phi^i\to0$   strongly in $ \mathrm{L}_{\al_i}(\rr^2) $)}
\end{align*}
and a contradiction arises!
\end{proof}

\section{A contraction mapping argument and the proof of the main theorem}\label{quattro}

First of all we point out  that $W_{\la}+\phi_{\la}$ is a solution to \eqref{p} if and only if
  $\phi_{\la}$ is a solution of the problem
\begin{equation}\label{L2}
 \mathcal{L}_{\la}(\phi )=\mathcal{N}_{\la}(\phi )+\mathcal{S}_{\la}\phi+\mathcal{R}_{\la} \  \hbox{in}\  \Omega \\\\
\end{equation}
where the error term $\mathcal{R}_\la$ is defined in \eqref{rla},
the linear error term  $\mathcal{S}_\la$ is defined in \eqref{sla}
the linear operator $\mathcal{L}_{\la}$ is defined in \eqref{lla} and
and the higher order term $\mathcal{N}_{\la}$ is defined as
\begin{equation}\label{nla} \mathcal{N}_{\la}(\phi ):=\la\[f\(W_{\la}+\phi \)-f\({W_{\la} }\)-f'\(  {W_{\la}}\)\phi\].\end{equation}

\begin{prop}\label{resto}
If  $p $ is close enough to $1$ there exist $\lambda_0>0$ and $R>0$
such that
  for any $\lambda \in (0,\lambda_0)$ there exists a unique solution $\phi_\lambda \in \mathrm{H}^1_0(\Omega)$ to
\begin{equation}\label{eqrid}-\Delta (W_\lambda+\phi_\lambda)=\lambda f(W_\lambda+\phi_\lambda)\ \hbox{in}\ \Omega,\quad \phi=0\ \hbox{on}\ \partial\Omega,\end{equation}
 such that $\phi(x)=\phi(-x)$ for any $x\in\Omega$ and
$$\|\phi_\lambda\|\le R\la^{2-p\over2p(2k-1)}|\ln\lambda|.$$
\end{prop}

\begin{proof}

 Let $\mathcal{H} := \{\phi\in\mathrm{H}^1_0(\Omega)\ :\ \phi(x)=\phi(-x)\ \forall\ x\in\Omega\}.$
As a consequence of Proposition \ref{inv}, we conclude that $\phi$
is a solution to \eqref{eqrid} if and only if it is a fixed point
for the operator $ \mathcal{T}_\lambda:
\mathcal{H}\to\mathcal{H},$ defined by
$$T_\lambda(\phi)= \left( \mathcal{L}_\lambda\right)^{-1}
\left(\mathcal{N}_\lambda(\phi)+\mathcal{S}_\lambda \phi +\mathcal{R}_\lambda\right),$$
where $\mathcal{L}_\lambda$, $\mathcal{N}_\la$, $\mathcal{S}_\la$ and $\mathcal{R}_\la$  are defined  in \eqref{rla},  \eqref{nla}, \eqref{sla} and \eqref{rla}, respectively.

 Let   us introduce the ball $B_{\la,R}:=\left\{\phi\in \mathcal{H} \ :\ \|\phi\|\le R\lambda^{ 2-p\over2p(2k-1)}\right\}$. We will show that $ T_\lambda:B_{\la,R}\to B_{\la,R}$ is a contraction mapping
provided $\la$ is small enough and $r$ is large enough.

\medskip {\em Let us prove that $T_\la$ maps  the ball $B_{\la,r}$ into itself,  i.e.}
\begin{equation}\label{c2.1}
\|\phi\|\le R\lambda^{2-p\over2p(2k-1)}|\ln\lambda|\ \Longrightarrow\
\left\|\mathcal{T}_\lambda(\phi)\right\|\le R\lambda^{2-p\over2p(2k-1)}|\ln\lambda|.
\end{equation}

By     Lemma \ref{B2} (where we take $h=\mathcal{N}_\lambda(\phi)+\mathcal{S}_\lambda \phi +\mathcal{R}_\lambda$), we deduce that:
\begin{align*}
\left\|\mathcal{T}_\lambda(\phi)\right\|&  \le c|\ln\la|  \left(\left\|\mathcal{N}_\lambda(\phi)\right\|_p+\left\|\mathcal{S}_\lambda \phi \right\|_p+\left\| \mathcal{R}_\lambda\right\|_p\right)\\   &\quad\hbox{(we use   \eqref{B21} with
 $p$ and $r$  close enough to $1 $ for the first term,}  \\   &\quad\hbox{we use   H\"older's inequality for the second term and}   \\   &\quad\hbox{we use    Lemma \ref{errore} for the third term)}\\
&  \le c|\ln\lambda|\( \|\phi\|^2e^{c_2\|\phi\|^2}\lambda^{ (2k-1){1-pr \over
pr}}+\left\|\mathcal{S}_\lambda \right\|_{pq} \left\|\phi \right\|_{ps} + \lambda^{ 2-p\over2p(2k-1)}\)\\
&  \le c|\ln\lambda|\( \|\phi\|^2e^{c_2\|\phi\|^2}\lambda^{ (2k-1){1-pr \over
pr}}+\lambda^{ 2-pq\over2pq(2k-1)} \left\|\phi \right\|  + \lambda^{ 2-p\over2p(2k-1)}\)\\   &\quad\hbox{(we use  Lemma \ref{tec1})}  \\   &
\end{align*}
and if we choose $q$ close enough to 1,    $R$  suitable large and $\la $   small enough we get \eqref{c2.1}.

\medskip {\em Let us prove that $T_\la$ is a contraction mapping,  i.e. there exists $L>1$ such that}
\begin{equation}\label{c2.2}
\|\phi\|\le R \lambda^{2-p\over2p(2k-1)}|\log\rho|\ \Longrightarrow\
\left\|\mathcal{T}_\lambda(\phi_1)-\mathcal{T}_\lambda(\phi_2)\right\|\le
L \|\phi_1-\phi_2\|.
\end{equation}

By     Lemma \ref{B2} (where we take $h=\mathcal{N}_\lambda(\phi_1 )-\mathcal{N}_\lambda(\phi_2 )+\mathcal{S}_\lambda (\phi_1-\phi_2)$), we deduce that:
\begin{align*}
&\left\|\mathcal{T}_\lambda(\phi)\right\|    \le c|\ln\la|  \left(\left\|\mathcal{N}_\lambda(\phi_1 )-\mathcal{N}_\lambda(\phi_2 )\right\|_p+\left\| \mathcal{S}_\lambda (\phi_1-\phi_2)\right\|_p\right)\\   &\quad\hbox{(we use   \eqref{B22} with
 $p$ and $r$  close enough to $1 $ for the first term and}  \\   &\quad\hbox{we use   H\"older's inequality for the second term)}   \\
&  \le c|\ln\la|  \[c_1e^{c_2(\|\phi_1\|^2+\|\phi_2\|^2)}\lambda^{  (2k-1){1-pr \over
pr}}
   \|\phi_1-\phi_2\|(\|\phi_1\|+\|\phi_2\|)\right.
   \\ &\left. \hskip2truecm+\left\|\mathcal{S}_\lambda \right\|_{pq} \left\|\phi_1-\phi_2 \right\|_{ps} \] \\
&  \le c|\ln\la|  \[c_1e^{c_2(\|\phi_1\|^2+\|\phi_2\|^2)}\lambda^{  (2k-1){1-pr \over
pr}}
   (\|\phi_1\|+\|\phi_2\|)\  +\lambda^{ 2-pq\over2pq(2k-1)}\]\left\|\phi_1-\phi_2 \right\|   \end{align*}
and if we choose $q$ close enough to 1,    $R$  suitable large and $\la $   small enough we get \eqref{c2.2}.

\end{proof}
  \begin{lemma}\label{B2}
  For any   $p\ge1$ and $r>1$ there exist $\lambda_0>0$ and $c_1,c_2>0$ such that for any   $\lambda \in (0,\lambda_0)$ we have for any $\phi,\phi_1,\phi_2\in\mathrm{H}^1_0(\Omega):$
 \begin{equation}\label{B21}
 \left\|\mathcal{N}_\lambda(\phi)\right\|_p\le c_1e^{c_2\|\phi \|^2}\lambda^{ (2k-1){1-pr \over pr}}\|\phi \|^2\end{equation}
  and
    \begin{equation}\label{B22}
\left\|\mathcal{N}_\lambda(\phi_1)-\mathcal{N}_\lambda(\phi_2)\right\|_p\le
c_1e^{c_2(\|\phi_1\|^2+\|\phi_2\|^2)}\lambda^{  (2k-1){1-pr \over
pr}}
   \|\phi_1-\phi_2\|(\|\phi_1\|+\|\phi_2\|).\end{equation}
  \end{lemma}
Proof Let us remark that (\ref{B21}) follows by   choosing
$\phi_2=0 $ in (\ref{B22}) . Let us prove (\ref{B22}). We point
out that
$$\mathcal{N}_\lambda(\phi)=\lambda e^{W_\lambda}\left(e^\phi-1-\phi\right)-\lambda e^{-W_\lambda}\left(e^{-\phi}-1+\phi\right)$$
 and so
$$\mathcal{N}_\lambda(\phi_1)-\mathcal{N}_\lambda(\phi_2)=
 \underbrace{\lambda e^{W_\lambda}  \left(e^{\phi_1}-e^{\phi_2}-\phi_1+\phi_2\right) }_{I_1}-\underbrace{\lambda e^{-W_\lambda}
 \left(e^{-\phi_1}-e^{-\phi_2}+\phi_1-\phi_2\right)}_{I_2}.$$
We estimate $\|I_1\|_p.$ The estimate of $\|I_2\|_p $ is similar.

 By the mean value theorem, we easily deduce that
 $$|e^a-e^b-a+b|\le e^{|a|+|b|}|a-b|(|a|+|b|)\ \hbox{for any }a,b\in\rr.$$
 Therefore,    we have
 \begin{eqnarray}\label{I1p}
 & &\|I_1\|_p=\left(\int\limits_\Omega \lambda^p e^{pW_\lambda}\left|e^{\phi_1}-e^{\phi_2}-\phi_1+\phi_2\right|^pdx\right)^{1/p}
 \nonumber\\ & &
 \le c\sum\limits_{j=1}^2
 \left(\int\limits_\Omega  \lambda^p e^{pW_\lambda} e^{p|\phi_1|+p|\phi_2|}|\phi_1-\phi_2|^p|\phi_j|^p dx\right)^{1/p}  \nonumber\\ & &\qquad \hbox{(we use H\"older's inequality  with ${1\over r}+{1\over s}+{1\over t}=1$ )}
 \nonumber\\ & &
 \le c\sum\limits_{j=1}^2
 \left(\int\limits_\Omega  \lambda^{pr} e^{prW_\lambda} dx\right)^{1/(pr)}
 \left(\int\limits_\Omega e^{ps|\phi_1|+ps|\phi_2|}dx\right)^{1/(ps)}
 \left(\int\limits_\Omega |\phi_1-\phi_2|^{pt}|\phi_j|^{pt} dx\right)^{1/(pt)}
  \nonumber\\ & &\qquad\hbox{(we use Lemma \ref{tmt})}
\nonumber\\ & &\
 \le c\sum\limits_{j=1}^2
 \left(\int\limits_\Omega  \lambda^{pr} e^{prW_\lambda} dx\right)^{1/(pr)}
  e^{(ps)/(8\pi)(|\phi_1|^2+|\phi_2|^2)}
  \|\phi_1-\phi_2\|\|\phi_j\|.\end{eqnarray}
We have to estimate
   \begin{align*}
  \int\limits_\Omega  \lambda^{pr} e^{prW_\lambda(x)} dx&=\sum\limits_{j}
 \int\limits_{A_j}  \lambda^{pr} e^{prW_\lambda(x)} dx,\end{align*}
 where $A_j$ is the annulus defined in \eqref{anelli}.

 If  $j$ is even  we get
  \begin{align*}&\int\limits_{A_j} \lambda^{pr}  e^{prW_\lambda(x)} dx\ \hbox{(we use \eqref{tetaj})}\\
  &=\delta_j^2\lambda^{pr} \int\limits_{A_j\over\delta_j} e^{pr \[w_j(\delta_jy)+(\alpha_j-2) \ln|\delta_jy|-\ln\lambda+\Theta_j(y)\]} dy\\ &
  = {\delta_j^{2-2pr} }\int\limits_{A_j\over\delta_j}\( 2\alpha_j^2{|y|^{\alpha_j-2}\over \(1+|y|^{\alpha_j}\)^2}\)^{pr}e^{pr \Theta_j(y) } dy\ \hbox{(we use Lemma \eqref{teta})} \\
  &=O\({\delta_j^{2-2pr} }\)=O\(\lambda^{(2k-1)(1-pr) }\)\ \hbox{( because $\delta_j\ge\delta_1=O\(\lambda^{2k-1\over2}\)$   and $pr>1 $)},\end{align*}
  and  $j$ is odd   we get
  \begin{align*}&\int\limits_{A_j}\lambda^{pr} e^{prW_\lambda(x)} dx\ \hbox{(we use \eqref{tetaj})}\\
   &=\delta_j^2\lambda^{pr}\int\limits_{A_j\over\delta_j} e^{-pr \[w_j(\delta_jy)+(\alpha_j-2) \ln|\delta_jy|-\ln\lambda+\Theta_j(y)\]} dy\\ &
  = {\delta_j^{2+2pr}\lambda^{2pr}}\int\limits_{\sqrt {\de_{j-1}\over\de_j}\le |y|\le \sqrt {\de_{j+1}\over\de_j}}\( {\(1+|y|^{\alpha_j}\)^2\over 2\alpha_j^2 |y|^{\alpha_j-2}  }\)^{pr}e^{-pr \Theta_j(y) } dy\ \hbox{(we use Lemma \eqref{teta})} \\
  &=O\({\delta_j^{2+2pr} \lambda^{2pr}}   \[\({\de_{j+1}\over\de_j}\)^{pr{\al_j+2\over2}+1}+\({\de_j\over\de_{j-1}}\)^{pr{\al_j-2\over2}-1}\]\)\nonumber\\
 = &O\( \lambda^{ pr} \[{\de_{j}^{pr+1}\over\de_{j+1}^{pr-1}}+{\de_{j}^{pr+1}\over\de_{j-1}^{pr-1}}\]\)= O\( \lambda^{ ({2k\over3}-1)(1-pr) }\)=O\( \lambda^{ ({2k }-1)(1-pr) }\),
  \end{align*}
  because by \eqref{delta3} and \eqref{alfa1} we get
  $$\({\de_{j+1}\over\de_j}\)^{pr{\al_j+2\over2}+1}= {\de_{j+1}^{pr {\al_{j+1}\over2 } }\over\de_j^{pr{\al_j\over2}}} {\de_{j+1}^{pr{\al_j-\al_{j+1}\over2}+pr+1}\over \de_j^{pr+1}}=O\( {1\over\la^{pr}\delta_j^{pr+1}\delta_{j+1}^{pr-1}}\),$$

 $$\({\de_j\over\de_{j-1}}\)^{pr{\al_j-2\over2}-1}= {\de_{j }^{pr {\al_{j }\over2}  }\over\de_{j-1}^{pr{\al_{j-1}\over2}}} {\de_j^{-pr-1}\over \de_{j-1}^{pr{\al_j-\al_{j-1}\over2}-pr-1} }=O\( {1\over\la^{pr}\delta_j^{pr+1}\delta_{j-1}^{pr-1}}\).$$
$$ {\de_{j}^{pr+1}\over\de_{j+1}^{pr-1}}=\({\de_{j} \over\de_{j+1}}\)^{pr } \de_j\de_{j+1}=o(1)$$
and
 \begin{align*}&\lambda^{ pr} {\de_{j}^{pr+1}\over\de_{j-1}^{pr-1}}= \lambda^{ pr}  \({\de_{j}\over\de_{j-1} }\)^{pr-1}\delta_j^2=O\( \lambda^{ pr}
 \({\de_{2}\over\de_{ 1} }\)^{pr-1}\delta_k^2\)\\ &=O\( \lambda^{ pr+{2k\over3}(1-pr)+{1\over 2k-1} }\)=O\( \lambda^{ ({2k\over3}-1)(1-pr) +{2k\over 2k-1}
 }\).
  \end{align*}
Therefore, by \eqref{I1p} we obtain that $\|I_1\|_p$ satisfies estimate \eqref{B22}.

   We recall the following Moser-Trudinger inequality \cite{Moe,Tru},
 \begin{lemma}\label{tmt} There exists $c>0$ such that for any   bounded domain $\Omega$ in $\rr^2$
 $$\int\limits_\Omega e^{4\pi u^2/\|u\|^2}dx\le c|\Omega|,\ \hbox{for any}\ u\in{\rm H}^1_0(\Omega).$$
 In particular,  there exists $c>0$ such that for any $\eta\in\rr$
 $$\int\limits_\Omega e^{\eta u}\le c|\Omega| e^{{\eta^2\over 16\pi}\|u\|^2},\
 \hbox{for any}\ u\in{\rm H}^1_0(\Omega).$$
 \end{lemma}

\begin{proof}[Proof of Theorem \ref{main}]
By Proposition \ref{resto} we have that
\begin{equation}\label{b1}
u_\la=W_\la+\phi_\la=\sum\limits_{i=1}^k
(-1)^iP  w_i(x)+\phi_\la
\end{equation}
is a solution to \eqref{p}.

\medskip
Let us prove \eqref{a1}.
By \eqref{pro-exp}, we derive that
\begin{equation}\label{b2}
P w_i(x)=  4\pi\al_iG(x,0)+o\(1\)\ \hbox{pointwise in }  \Omega\setminus\{0\}, 
\end{equation}
 and so, by  \eqref{b1} and \eqref{somma} we get,
\begin{equation}\label{b3}
u_\la(x)=4\pi\sum\limits_{i=1}^k(-1)^i\al_iG(x,0)+o\(1\)=(-1)^k8\pi kG(x,0)+o\(1\)\ \hbox{pointwise in }  \Omega\setminus\{0\}
\end{equation}
  Moreover, for some $\theta\in(0,1)$ we have that
\begin{equation}\label{b4}
\la e^{W_\la+\phi_\la}-\la e^{-W_\la-\phi_\la}=\la e^{W_\la}-\la e^{-W_\la}+\la e^{W_\la+\theta\phi_\la}\phi_\la
+\la e^{-W_\la-\theta\phi_\la}\phi_\la.
\end{equation}
Let us fix a compact set $K\subset\Omega$ which does not contain the origin and let $q>1.$ From  \eqref{b4} we deduce
\begin{align}\label{b5}
 & \|\la e^{W_\la+\phi_\la}-\la e^{-W_\la-\phi_\la}\|_{L^q(K)} \nonumber\\
&=O\( \|\la e^{W_\la}-\la e^{-W_\la}\|_{L^q(K)}\)  +O\(\|\la e^{W_\la+\theta\phi_\la}\phi_\la+\la e^{-W_\la-\theta\phi_\la}\phi_\la\|_{L^q(K)} \)\nonumber\\
&=O(1) \ \hbox{(because of the definition of $W_\la$ and \eqref{b2})}\nonumber\\
&+O\(\|\la e^{W_\la}\|_{L^{\infty}(K)})\|\la e^{\theta\Phi_\la}\|_{L^{2q}(K)}\|\phi\|_{L^{2q}(K)}\)
+O\(\|\la e^{-W_\la}\|_{L^{\infty}(K)}\|\la e^{-\theta\Phi_\la}\|_{L^{2q}(K)}\|\phi\|_{L^{2q}(K)}\)\nonumber\\
&=O(1) \ \hbox{(using \eqref{b2}, Lemma \ref{tmt} and Proposition \ref{resto})}.\end{align}
Hence, by  \eqref{b4} and  \eqref{b5}, we derive that the R.H.S. of problem \eqref{p} is bounded in $L^q(K)$. Then standard results imply \eqref{a1}.

\medskip
Let us prove  \eqref{a2} and  \eqref{a3}. We will only consider the case $m_+(0)$ , since the other is similar. By the definition of $m_+(0)$ we get
\begin{align*}
m_+(0)=&\lim\limits_{r\to0}\lim\limits_{\la\to0}\int\limits_{B(0,r)}\la e^{W_\la+\phi_\la}=\lim\limits_{r\to0}\lim\limits_{\la\to0}\int\limits_{B(0,r)}\la e^{W_\la}(1+e^{\theta\phi_\la}\phi_\la)\\
=&\lim\limits_{r\to0}\lim\limits_{\la\to0}\(J_{1,\la,r}+J_{2,\la,r}\).
\end{align*}
So we have that,
\begin{align*}
 J_{1,\la,r}&=\int\limits_{B(0,r)}\la e^{W_\la}=\int\limits_{B(0,r)}\la e^{\sum\limits_{i=1}^k
(-1)^iP w_i(x)}=\sum\limits_{j=1}^k \int\limits_{A_j}\la e^{(-1)^jPw_j(x)+\sum\limits_{i=1\atop
i\not= j}^k(-1)^{i-j}P{w_i(x)}}dx \\
&\qquad\hbox{(using \eqref{anelli})}\nonumber\\
& =\sum\limits_{j=1\atop j\ {\rm even}}^k \int\limits_{A_j}\la e^{Pw_j(x)+\sum\limits_{i=1\atop
i\not= j}^k(-1)^{i-j}P{w_i(x)}}  dx+\sum\limits_{j=1\atop j\ {\rm odd}}^k \int\limits_{A_j}\la e^{-Pw_j(x)+\sum\limits_{i=1\atop
i\not= j}^k(-1)^{i-j}P{w_i(x)}}dx\\
&= 
\sum\limits_{j=1\atop j\ {\rm even}}^k \int\limits_{\rr^2}|x|^{\al_j-2}e^{w_j(x)}dx+o(1)\  \hbox{(using \eqref{er3b} and \eqref{er5} with $p=1$)} \\ 
&=\sum\limits_{j=1\atop j\ {\rm even}}^k 4\pi\al_i+o(1) \ \hbox{(we use \eqref{mass})}\\ &=
\sum\limits_{j=1\atop j\ {\rm even}}^k 4\pi(4i-2)+o(1) \ \hbox{(we use  \eqref{alfa})}\\
& = \left\{\begin{aligned}
&4\pi k(k-1)\quad\hbox{if $k$ is even}\\
&4\pi k(k+1)\quad\hbox{if $k$ is odd }
 \end{aligned}\right. +o(1).
\end{align*}
Here we used a result of Chen-Li \cite{cl} which states the mass
\begin{equation}\label{mass}
\int\limits_{\rr^2} |y|^{\al-2}e^{w^\al(x)}dx=4\pi\al,\
\end{equation}
where
$$w^\al (x):=\ln 2\al^2{1\over\(1+|x|^\al\)^2},\quad
x\in\rr^2.
$$
 So $$\lim\limits_{r\to0}\lim\limits_{\la\to0}I_{1,\la,r}=\left\{\begin{aligned}
&4\pi k(k-1)\quad\hbox{if $k$ is even}\\
&4\pi k(k+1)\quad\hbox{if $k$ is odd}
 \end{aligned}\right. .$$
On the other hand, arguing exactly as in \eqref{I1p}
we get (for some ${1\over p}+{1\over q}+{1\over s}=1$)
\begin{align*}
 J_{2,\la,r}&=\int\limits_{B(0,r)}\la e^{W_\la}e^{\theta\phi_\la}\phi_\la=O\(\|\la e^{W_\la}\|_p\|\la e^{\theta\Phi_\la}\|_{q}\|\phi\|_{s}\)=o(1) .
\end{align*}
So we have that $$\lim\limits_{r\to0}\lim\limits_{\la\to0}I_{2,\la,r}=0$$ which ends the proof.
\end{proof}

\section{Appendix}
We have the following result.

\begin{thm}
\label{esposito} Assume $\alpha\over2$  is odd .
If $\phi  $  satisfies 
\begin{equation}\label{even}\phi(y)=\phi(-y)\ \hbox{for any     $y\in\rr^2$}\end{equation} and
solves the equation
\begin{equation}\label{l1}
-\Delta \phi =2\alpha^2{|y|^{\alpha-2}\over (1+|y|^\alpha)^2}\phi\ \hbox{in}\ \rr^2,\quad \int\limits_{\rr^2}|\nabla \phi(y)|^2dy<+\infty,
\end{equation}
then there exists $a\in\rr$ such that
$$\phi(y)=\gamma   {1-|y|^\al\over 1+|y|^\al}\ \hbox{for some}\ \gamma\in\rr .$$
\end{thm}
\begin{proof}
Del Pino-Esposito-Musso in \cite{dem} proved that all the bounded solutions to \eqref{l1} are a linear combination of the following functions (which are written in polar coordinates)
$$\phi_0(y):   ={1-|y|^\alpha\over 1+|y|^\alpha} ,\  \phi_1(y):={ |y|^{\alpha\over 2}\over 1+|y|^\alpha} \cos{\alpha\over2}\theta 
,\
\phi_2(y):={ |y|^{\alpha\over 2}\over 1+|y|^\alpha} \sin{\alpha\over2}\theta.
$$
We observe that $\phi_0$   always satisfies \eqref{even}, while if $\alpha\over2$  is odd the functions $\phi_1$ and $\phi_2$ do not satisfy \eqref{even}.
So, we just have  to prove that any solution $\phi$ of \eqref{l1} is actually a bounded solution, i.e. $\phi\in \mathrm{L}^\infty(\rr^2).$ The claim will follow by \cite{dem}.\\ 
Since $ \phi$ is a solution in the sense of distribution to \eqref{l1}, from the boundedness of RHS in $L^2_{loc}(\rr^2)$ and by the regularity theory we get that $ \phi\in L^\infty_{loc}(\rr^2)$.\\
In order to end the proof we have to show that $ \phi$ is bounded near infinity. Let us consider the Kelvin transform of $ \phi$, namely,
$$z(x)= \phi\left(\frac x{|x|^2}\right).$$
A straightforward computation gives
$$\int\limits_{\rr^2}|\nabla z(y)|^2dy=\int\limits_{\rr^2}|\nabla \phi(y)|^2dy$$
and
$$
-\Delta z =2\alpha^2{|y|^{\alpha-2}\over (1+|y|^\alpha)^2}z\ \hbox{in}\ \rr^2.
$$
So we have that $z$ satisfies the same problem as $\phi$ and then $z\in L^\infty_{loc}(\rr^2)$. This implies that $\phi$ is bounded near infinity which ends the proof.
\end{proof}

 For any $\alpha\ge2$  let us consider the Banach spaces
\begin{equation}\label{ljs}
\mathrm{L}_\alpha(\rr^2):=\left\{u \in {\rm W}^{1,2}_{loc}(\rr^2)\ :\  \left\|{|y|^{\alpha-2\over 2}\over 1+|y|^\alpha}u\right\|_{\mathrm{L}^2(\rr^2)}<+\infty\right\}\end{equation}
 and
\begin{equation}\label{hjs}\mathrm{H}_\alpha(\rr^2):=\left\{u\in {\rm W}^{1,2}_{loc}(\rr^2) \ :\ \|\nabla u\|_{\mathrm{L}^2(\rr^2)}+\left\|{|y|^{\alpha-2\over 2}\over 1+|y|^\alpha}u\right\|_{\mathrm{L}^2(\rr^2)}<+\infty\right\},\end{equation}
 endowed with the norms
$$\|u\|_{\mathrm{L}_\alpha}:= \left\|{|y|^{\alpha-2\over 2}\over 1+|y|^\alpha}u\right\|_{\mathrm{L}^2(\rr^2)}\ \hbox{and}\ \|u\|_{\mathrm{H}_\alpha}:= \(\|\nabla u\|^2_{\mathrm{L}^2(\rr^2)}+\left\|{|y|^{\alpha-2\over 2}\over 1+|y|^\alpha}u\right\|^2_{\mathrm{L}^2(\rr^2)}\)^{1/2}.$$
\begin{prop}\label{compact}
The embedding $i_\al:\mathrm{H}_\alpha(\rr^2)\hookrightarrow\mathrm{L}_\alpha(\rr^2)$
is compact.
\end{prop}
\begin{proof}
Firstly, let $\alpha=2.$ If $\mathbb{S}^2$ denotes the unit sphere
in $\rr^3$ with the standard metric and $\pi:\mathbb{S}^2\to\rr^3$
is the stereographic projection through the north pole, then the
map $u\to u \circ\pi$ is an isometry from $\mathrm{L}_ 2$ into
$\mathrm{L}^2(\mathbb{S}^2)$ and from $\mathrm{H}_ 2$ into
$\mathrm{H}^1(\mathbb{S}^2)$ Hence, the claim follows directly
from the compactness of the embedding
$\mathrm{H}^1(\mathbb{S}^2)\hookrightarrow
\mathrm{L}^2(\mathbb{S}^2).$ Now, let $\alpha\ge 2.$ Let us define
an operator $\mathfrak{T}_\alpha:\mathrm{L}_\alpha(\rr^2)\to\mathrm{L}_2(\rr^2)$
by
 $\mathfrak{T}_\alpha(u)=\overline u$ where the function $\overline u$ is defined in this way
 $$\overline u(z):=\hat u(|z|,\theta), \  \hbox{where}\  \hat u(s,\theta):=\tilde  u(s^{2/\alpha},\theta)\ \hbox{and}\  \tilde  u(r,\theta):=u \(r\cos\theta,r\sin\theta\) .$$

 We will prove that
 \begin{equation}\label{stima}\|\mathfrak{T}_\alpha\|_{\mathcal{L}\(\mathrm{L}_\alpha(\rr^2),\mathrm{L}_2(\rr^2)\)}={2\over\alpha}\quad \hbox{and}\quad
 {2\over\alpha}\le\|\mathfrak{T}_\alpha\|_{\mathcal{L}\(\mathrm{H}_\alpha(\rr^2),\mathrm{H}_2(\rr^2)\)}\le{\alpha\over2}.
 \end{equation}
 The compactness of the embedding for $\alpha\ge2$ will follow immediately, because $i_\al=\mathfrak{T}_\alpha^{-1}\circ i_2\circ\mathfrak{T}_\alpha.$

 \medskip
 Let us prove \eqref{stima}
  A direct computation shows that
\begin{align*}\int\limits_{\rr^2} {|y|^{\alpha-2 }\over (1+|y|^\alpha)^2}u^2(y)\ dy&= \int\limits^{2\pi}_0 \int\limits^{+\infty}_0{r^{\alpha-1}\over (1+r^\alpha)^2}\tilde u^2(r,\theta)\ dr \ d\theta ={2\over\alpha}
\int\limits^{2\pi}_0 \int\limits^{+\infty}_0{s \over (1+s^2)^2}\hat u^2(s,\theta)\ ds \ d\theta\\
&={2\over\alpha}\int\limits_{\rr^2} {1\over (1+|z|^\alpha)^2}
\overline u^2(z)\ dz,\end{align*} which proves the first estimate
in \eqref{stima}. Moreover, we also have
\begin{align*}\int\limits_{\rr^2} |\nabla u|^2(y)\ dy&= \int\limits^{2\pi}_0 \int\limits^{+\infty}_0 r \[\(\partial_r \tilde u\)^2+
{\(\partial_\theta \tilde u \)^2\over r^2}
\]\ dr \ d\theta ={2\over\alpha}\int\limits^{2\pi}_0 \int\limits^{+\infty}_0 s \left\{{\alpha^2\over4}\(\partial_s \hat u\)^2+
{\(\partial_\theta \hat u \)^2\over s^2} \right\}\ ds \
d\theta\end{align*} and so
$${2\over\alpha}\int\limits_{\rr^2} |\nabla \overline u|^2(z)\ dz\le \int\limits_{\rr^2} |\nabla u|^2(y)\ dy\le {\alpha\over2}\int\limits_{\rr^2} |\nabla \overline u|^2(z)\ dz,$$ which proves the second estimate in \eqref{stima}.

\end{proof}

\end{document}